\documentclass[a4paper]{amsart}
                   
\usepackage{a4wide}

\usepackage{graphicx}
\usepackage{tikz,pgfplots}
\usepackage{amssymb}
\usepackage{pdfsync}
\usepackage{caption}
\usepackage{hyperref}
\usepackage{verbatim} 
\usepackage{epstopdf}
\usepackage[normalem]{ulem}

\def\R{\mathbb{R}}
\def\T{\mathbb{T}}
\def\N{\mathbb{N}}
\def\Z{\mathbb{Z}}

\DeclareMathOperator{\sgn}{sign}

\newtheorem{theorem}{Theorem}

\newtheorem{proposition}[theorem]{Proposition}
\newtheorem{lemma}[theorem]{Lemma}

\newtheorem{example}[theorem]{Example}

\begin{document}

\title[Hermite polynomials, linear flows and uncertainty principles]{Hermite polynomials, linear flows on the torus,\\ and an uncertainty principle for roots}

\author{Felipe Gon\c{c}alves}
\address{
        Felipe Gon\c{c}alves\\
        IMPA, Estrada Dona Castorina 110\\
        Rio de Janeiro, RJ  22460-320, Brazil}
\email{ffgoncalves@impa.br}

\author{Diogo Oliveira e Silva}
\address{
        Diogo Oliveira e Silva\\
        Hausdorff Center for Mathematics\\
        53115 Bonn, Germany}
\email{dosilva@math.uni-bonn.de}

\author{Stefan Steinerberger}
\address{
        Stefan Steinerberger\\
Department of Mathematics\\
        Yale University\\
        New Haven, CT 06511, USA}
\email{stefan.steinerberger@yale.edu}

\begin{abstract} 
We study a recent result of Bourgain, Clozel and Kahane, a version of which states that a sufficiently nice function $f:\mathbb{R} \rightarrow \mathbb{R}$ that coincides with its Fourier transform and vanishes at the origin has a root in the interval $(c, \infty)$, 
where the optimal $c$ satisfies $0.41 \leq c \leq 0.64$. A similar result holds in higher dimensions. 
We improve the one-dimensional result to $0.45 \leq c \leq 0.594$, and the lower bound in higher dimensions.
We also prove that extremizers exist, and have infinitely many double roots. With this purpose in mind, we establish a new structure statement about Hermite polynomials which relates their pointwise evaluation to linear flows on the torus, and applies to other families of orthogonal polynomials as well.
\end{abstract}


\maketitle

\vspace{-20pt}

\section{Introduction and main results}
Throughout the paper, we will use the normalization that turns the Fourier transform into a unitary operator on $L^2(\mathbb{R}^d)$:

\begin{equation}\label{FT}
\widehat{f}(y)=\int_{\R^d} f(x) e^{-2\pi i x\cdot y} dx
\end{equation}

\subsection{Setup}
The following  insight is due to Bourgain, Clozel and Kahane \cite{BCK}:
If $f:\mathbb{R} \rightarrow \mathbb{R}$ is an even function such that  $f(0) \leq 0$ and $\widehat{f}(0) \leq 0$, then it is not possible for both $f$ and
$\widehat{f}$ to be positive outside an arbitrarily small neighborhood of the origin. 
Having $f$ even and real-valued guarantees that $\widehat{f}$ is real-valued and even. {The second condition yields
$$ 0 \geq \widehat{f}(0)  = \int_{-\infty}^\infty{f(x)dx} \qquad \mbox{and} \qquad 0 \geq f(0) = \int_{-\infty}^\infty{\widehat{f}(y)dy},$$
which implies that the quantities
$$A(f):=\inf~\{r>0: f(x)\geq 0 \textrm{ if } |x|>r\}$$
$$A(\widehat{f}):=\inf~\{r>0: \widehat{f}(y)\geq 0 \textrm{ if } |y|>r\}$$
are strictly positive (possibly $\infty$) unless $f \equiv 0$. There is a dilation
symmetry $x \rightarrow \lambda x$ having the reciprocal effect $y \rightarrow y/\lambda$ on the Fourier side. As a consequence, the product $A(f)A(\widehat{f})$ is
invariant under this group action and becomes a natural quantity to consider.

\subsection{One-dimensional bounds} The paper \cite{BCK} establishes the following quantitative result.

\begin{theorem}[Bourgain, Clozel \& Kahane]\label{BCK1}
Let $f:\R\to\R$ be a nonzero, integrable, even function such that $f(0) \leq 0$, $\widehat{f} \in L^1(\mathbb{R})$ and $\widehat{f}(0) \leq 0$. Then
$$A(f)A(\widehat{f}) \geq 0.1687,$$
and $0.1687$ cannot be replaced by $0.41$.
\end{theorem}

It is straightforward to prove \textit{some} lower bound for the quantity $A(f)A(\widehat{f})$, see Lemma \ref{trivialLB} below for a very short and easy proof taken from \cite{BCK} of the lower bound $1/16$.
The purpose of the present paper is to popularize the statement,
to give new proofs of improved estimates, and to investigate properties of extremizers. 
Our first argument improves the constants.

\begin{theorem}\label{1Dthm}
Let $f:\R\to\R$ be a nonzero, integrable, even function such that  $f(0) \leq 0$, $\widehat{f} \in L^1(\mathbb{R})$ and $\widehat{f}(0) \leq 0$. Then
$$A(f)A(\widehat{f})\geq 0.2025, $$
and $0.2025$ cannot be replaced by $0.353$.
\end{theorem}  

The proof of the lower bound in Theorem \ref{1Dthm} 
relies on rearrangement inequalities of optimal transport flavor which do not admit a straightforward generalization to higher dimensions.
It is quite involved and cannot be improved much further: the third decimal place in the lower bound could be increased at the expense of some additional work, but a genuinely new idea seems needed for substantial further improvement. In contrast, we believe that the upper bound given by Theorem \ref{1Dthm} 
 might be very close to being optimal and that functions which almost realize the sharp constant look like the function depicted in Figure \ref{fig:goodcandidate}. 
 \begin{figure}[h!]
\includegraphics[width = 0.55\textwidth]{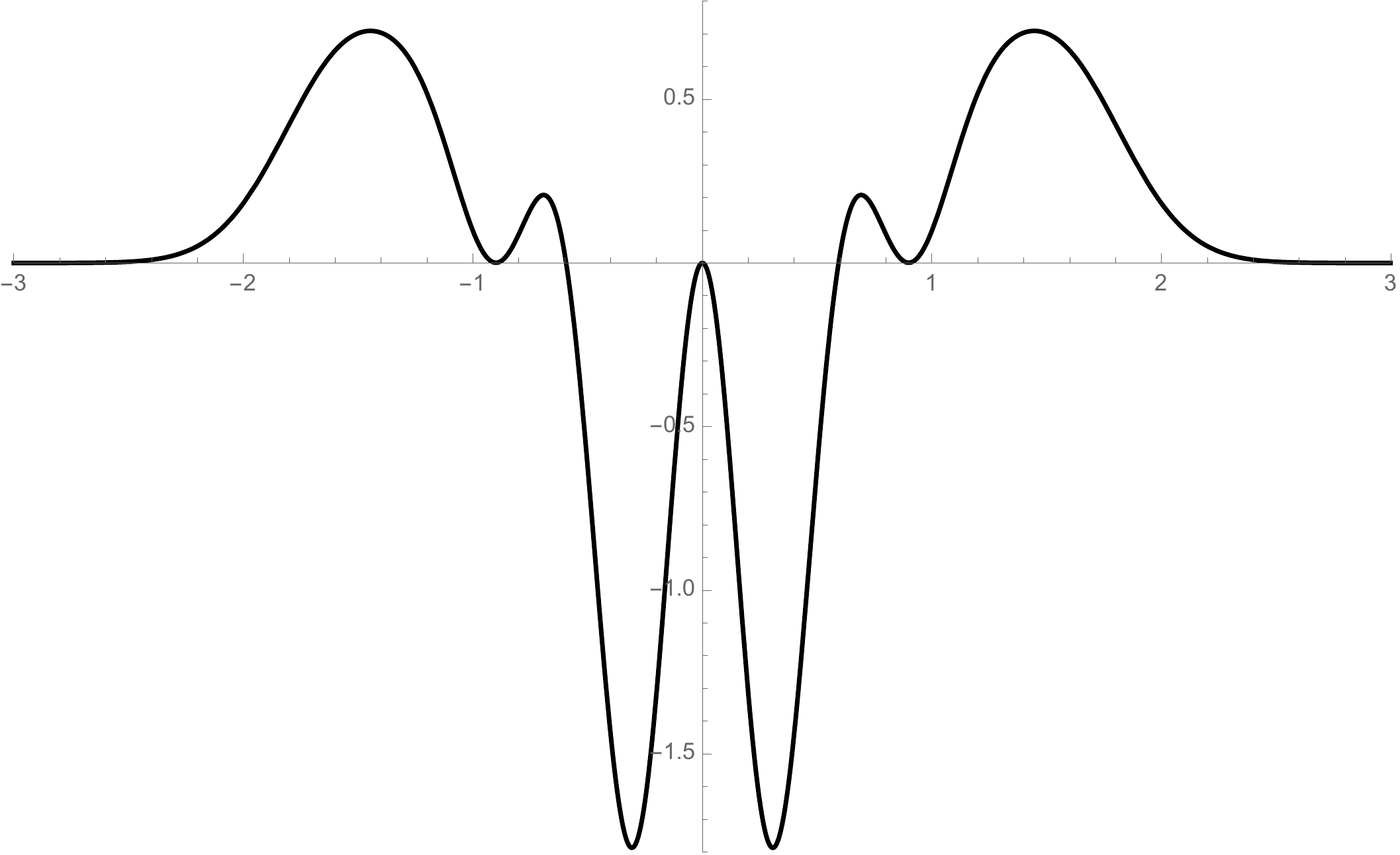}
\caption{Plot of a function $f\in L^1(\R)$ satisfying  $\widehat{f} = {f}$ and $f(0) = 0$ which  is non-negative in the interval $(0.6, \infty)$.}
\label{fig:goodcandidate}
\end{figure}

\subsection{Extremizers}
Let $\mathcal{A}$ denote the higher-dimensional version of the set of functions considered in Theorems \ref{BCK1} and \ref{1Dthm}. In other words, let $d\geq 1$, and say that a function $f:\R^d\to\R$ belongs to $\mathcal{A}$ if it is nonzero, integrable with integrable Fourier transform, and such that $f(0)\leq 0$ and $\widehat{f}(0)\leq 0$. Set
$${\bf A}:=\inf_{f\in\mathcal{A}} \sqrt{A(f)A(\widehat{f})},$$
where $A(f)$ again denotes the smallest positive real number $r$ such that $f(x)\geq 0$, for every $|x|>r$. 
Our next result  shows that the inequality 
\begin{equation}\label{sharp}
A(f)A(\widehat{f})\geq {\bf A}^2\;\;\; (f\in\mathcal{A})
\end{equation}
admits an extremizer. It holds in every dimension $d\geq 1$.

\begin{theorem}\label{existence-thm}
There exists a nonzero radial function $f\in L^1(\R^d)$ such that $\widehat f=f$, $f(0)=\widehat f(0)=0$, and $A(f)={\bf A}$.
\end{theorem}

We proceed to show that extremizers for inequality \eqref{sharp} exhibit an unexpected behavior when compared to extremizers for other uncertainty principles (recall, for instance, that Gaussians extremize the Heisenberg uncertainty inequality). To state it precisely, let us say that a continuous function $f:\R\to\R$ has a {\it double root} at $x_0\in\R$ if $f(x_0)=0$ and $f$ does not change sign in a neighborhood of $x_0$.

\begin{theorem}\label{ext}  
 Let $f:\R_{+}\to \R$ be a function such that its radial extension, $x\in\R^d \mapsto f(|x|)$, belongs to the set $\mathcal{A}$ and realizes equality in \eqref{sharp}. Then 
$f$ has infinitely many double roots in the interval $(A(f), \infty)$.
\end{theorem}

We remark that, in principle, it is possible for an extremizer $f$ to vanish identically in an interval $[a,b]\subset [A(f),\infty)$ and to be strictly positive for large values of its argument, although we believe that not to be the case. 
We approach Theorem \ref{ext} in two different ways, both of which follow a common  general strategy: Assuming $f$ to be an extremizer for inequality \eqref{sharp} with a finite number of double roots only, we identify a perturbation $f_\varepsilon$ of $f$ for which
$ A(f_\varepsilon)A(\widehat{f_\varepsilon}) < A(f)A(\widehat{f}).$
The first argument works only if $d=1$, but has the advantage that it relies on an explicit construction of the perturbation $f_\varepsilon$ that seems generalizable to a number of related situations which we plan to address in future work.
This construction makes use of a variant of the following nice result about Hermite polynomials which holds at a greater level of generality, and may be true for a wide class of orthogonal functions.

\begin{theorem}\label{Hermite} 
Let $ \left\{a_1, a_2, \dots, a_k\right\} \subset \mathbb{R}$ be a finite set of reals. Then there exist infinitely many Hermite polynomials $H_{4n}$ satisfying 
$$ \min_{1 \leq j \leq k}{H_{4n}(a_j) } > 0,$$
and there exist infinitely many Hermite polynomials $H_{4n+2}$ satisfying 
$\max_{1 \leq j \leq k}{H_{4n+2}(a_j) } < 0.$
\end{theorem}
Variants of this statement should  hold 
for `generic' families of orthogonal functions. In fact, we prove similar results for Laguerre polynomials, as well
as for certain linear combinations of Hermite polynomials that appear naturally in the one-dimensional proof of Theorem \ref{ext}. 
We believe this question, namely, to which extent do sequences of orthogonal functions realize particular sign patterns when simultaneously evaluated at a prescribed finite set of distinct  points, to be of independent interest and further comment on it below. The second part of the proof of Theorem \ref{ext} works only in higher dimensions $d\geq 2$, and makes use of Laguerre expansions of radial functions.

\subsection{Bounds in higher dimensions}
A version of Theorem \ref{BCK1} holds in higher dimensions.

\begin{theorem}[Bourgain, Clozel \& Kahane]  \label{BCKhigherD}
Let $d\geq 2$. Let $f \in L^1(\mathbb{R}^d)$ be a nonzero, real-valued,  radial function such that $f(0) \leq 0 $, $\widehat{f} \in L^1(\mathbb{R}^d)$ and $\widehat{f}(0) \leq 0$. Then
$$A(f)A(\widehat{f})  \geq \frac{1}{\pi}\Big(\frac1 2 \Gamma\Big(\frac d 2 +1\Big)\Big)^{\frac 2 d},$$
and this lower bound cannot be replaced by $(d+2)/2\pi$.
\end{theorem}
As an immediate consequence, we have 
\begin{equation}\label{lineargrowth}
\frac{d}{2\pi e}<\inf_{f} A(f)A(\widehat{f}) < \frac{d+2}{2\pi},
\end{equation}
where the infimum is taken over all functions $f$ satisfying the assumptions of Theorem \ref{BCKhigherD}. The linear growth in terms of dimension given by inequalities \eqref{lineargrowth} is expected in a wider class of related situations. 
The last chapter of the paper \cite{BCK} shows that this problem and its solution are naturally related to the theory of zeta-functions in algebraic number fields.
 Arithmetic arguments show that the linear growth of the bounds with respect to dimension is natural in view of known properties of ramifications of these fields. 
We show that a variation of the original argument employed in \cite{BCK} to handle the one-dimensional case can be used to improve the lower bound in all higher dimensions. 
\begin{theorem}\label{Dthm} 
Let $d\geq 2$. Let $f \in L^1(\mathbb{R}^d)$ be a nonzero real-valued,  radial function such that $f(0) \leq 0 $, $\widehat{f} \in L^1(\mathbb{R}^d)$ and $\widehat{f}(0) \leq 0$. Then:
$$A(f)A(\widehat{f}) \geq \frac{1}{\pi}\Big(\frac1 {1+\lambda_d} \Gamma\Big(\frac d 2 +1\Big)\Big)^{\frac 2 d},$$
where the number $\lambda_d$ is defined in terms of the Bessel function $J_{d/2}$ as
$$\lambda_d:=-\inf_{u\in\R_+}\frac{\Gamma\Big(\frac d 2+1\Big) J_{ d/2}(u)}{(u/2)^{d/2}}.
$$
Moreover, $\lambda_d<\frac1 2$ for every $d\geq 2$, and $\lambda_d\to 0$ as $d\to\infty$ exponentially fast.
\end{theorem}

\subsection{Overview}
The paper is organized as follows. We gather relevant information about Hermite functions, Bessel functions and Laguerre polynomials in \S \ref{sec:PreApp}, together with a brief digression on one-dimensional rearrangements of functions. We perform a number of elementary reductions in \S \ref{sec:Reductions}, and establish the aforementioned lower bound of 1/16 in Lemma \ref{trivialLB} below. We prove Theorem \ref{1Dthm} in \S \ref{sec:1DThm}. We proceed in two steps, first proving the lower bound and then establishing the upper bound via an explicit example. The next \S \ref{sec:flows} is devoted to the study of linear flows on the torus. In particular, we establish a result that will play a role in the one-dimensional proof of Theorem \ref{ext}, and additionally prove Theorem \ref{Hermite}. Extremizers for inequality \eqref{sharp} are studied in \S \ref{sec:Structure}, where we prove Theorems \ref{existence-thm} and \ref{ext}. Finally, \S \ref{sec:HigherDim} is devoted to the proof of Theorem \ref{Dthm}.\\

\noindent{\bf Acknowledgements.} The authors are grateful to Ronald R. Coifman, Jo\~ao Pedro Ramos and Christoph Thiele for various useful comments and suggestions. F.G. is supported by CNPQ-Brazil Post-Doctoral Junior Fellowship 150386/2016-8, D.O.S. is supported by the Hausdorff Center for Mathematics, and S.S. is supported by an AMS-Simons Travel Grant and INET Grant \#INO15-00038.
This work was started during a pleasant visit of the third author to the Hausdorff Institute for Mathematics, whose hospitality is greatly appreciated. 

\section{Special functions, rearrangements and integrals over spheres}\label{sec:PreApp}
The purpose of this chapter is to collect various facts which will  appear in the arguments below in order to keep the paper as self-contained as possible.

\subsection{Hermite functions}\label{ssec:Hermite} 
The Hermite polynomials constitute an orthogonal family on the real line with respect to the Gaussian measure. They can be defined for $n\in\N$ and $x\in\R$ as follows:
$$
H_n(x):=(-1)^n e^{x^2}\frac{d^n}{dx^n}(e^{-x^2}).
$$
The orthogonality formula
\begin{equation}\label{HermiteOrthogonality}
\int_{-\infty}^\infty H_n(x)H_m(x) e^{-x^2} dx= 2^n n! \sqrt{\pi}\delta{(n-m)}
\end{equation}
can be checked via $\max\{m,n\}$ integrations by parts, or can be taken as an alternative definition as is done in \cite{Sz}. We use the following asymptotic expansion for Hermite polynomials \cite[Theorem 8.22.6 and (8.22.8)]{Sz}
\begin{equation}\label{HermiteAsympt}
\frac{\Gamma(n/2+1)}{\Gamma(n+1)}e^{-\frac{x^2}2} H_n(x) 
= \cos{  \left(\sqrt{2n+1}x - \frac{n \pi}{2} \right)} + \frac{x^3}{6}\frac{1}{\sqrt{2n+1}}\sin{\left(\sqrt{2n+1} x - \frac{n \pi}{2}\right)} + \mathcal{O}\left(\frac{1}{n}\right),
\end{equation}
which is valid for any fixed $x \in \mathbb{R}$ as $n\to\infty$. Indeed, as pointed out in \cite{Sz}, the result holds on compact intervals with a uniformly bounded constant in the error term. For all but one application, the simpler
expansion
\begin{equation}\label{simplifiedHermiteAsympt}
\frac{\Gamma(n/2+1)}{\Gamma(n+1)}e^{-\frac{x^2}2} H_n(x) 
= \cos{  \left(\sqrt{2n+1}x - \frac{n \pi}{2} \right)}  + \mathcal{O}\left(\frac{1}{\sqrt{n}}\right)
\end{equation}
will suffice. 
The rescaled Hermite functions 
$$\psi_n(x):= \frac{2^{1/4}}{\sqrt{2^n n!}}H_n(\sqrt{2\pi} x) e^{-\pi x^2}$$
form an orthonormal basis of $L^2(\R)$ and are a set of eigenfunctions for the Fourier transform normalized as in \eqref{FT}. More precisely, we have that
$$\widehat{\psi_n}=(-i)^{n(\textrm{mod} 4)}\psi_n.$$
In particular, a function $f\in L^2(\R)$  equals its own Fourier transform if and only if it admits an expansion of the form
\begin{equation}\label{HermiteExpansion}
f(x)=\sum_{n=0}^\infty a_n \psi_{4n}(x)
\end{equation}
for a (necessarily unique) set of coefficients $\{a_n\}\subset\ell^2(\N)$.

 \subsection{Gamma function}
The Gamma function is defined for $\Re(s)>0$ as
\begin{equation}\label{def_Gamma}
\Gamma(s)=\int_0^\infty e^{-t} t^{s-1} dt.
\end{equation}
It satisfies the functional equation $s\Gamma(s)=\Gamma(s+1)$ and thus constitutes a meromorphic extension of the factorial: $\Gamma(n+1)=n!$ for every $n\in\N$. The following version of Stirling's formula \cite{R} will be useful. For every $x\geq 0$,
 \begin{equation}\label{Stirling}
 \Gamma(x)=\sqrt{2\pi}x^{x-1/2}e^{-x}e^{\mu(x)} \qquad \mbox{where} \quad \frac1{12x+1}<\mu(x)<\frac1{12x}.
 \end{equation}
 
 \subsection{Bessel functions} 
The Bessel function of the first kind $J_\nu$ can be defined in a number of ways. We follow the treatise \cite{W} and define it for $\nu > -1$ and $\Re(z) >0$ by 
\begin{equation}\label{def_Bessel}
J_\nu(z) = \Big(\frac z 2\Big)^\nu \sum_{n=0}^{\infty} \frac{(-1)^n \big(\tfrac z 2\big)^{2n}}{n!\, \Gamma(\nu+n+1)}.
\end{equation}
One can check that Bessel functions satisfy the  differential equation
\begin{equation}\label{BesselODE}
z^2J''_\nu(z)+z J_\nu'(z)+(z^2-\nu^2)J_\nu(z)=0,
\end{equation}
and that the following recursion relations hold
\begin{align}
J_{\nu-1}(z)-J_{\nu+1}(z)&=2 J'_\nu(z),\label{Recurrence1}\\
J_{\nu-1}(z)+J_{\nu+1}(z)&=\frac{2\nu}{z} J_\nu(z).\label{Recurrence2}
\end{align}
An alternative definition of the Bessel functions, valid for all values of $\nu>-1/2$, is contained in the following Poisson integral representation:
\begin{equation}\label{alt_Bessel}
J_\nu(z)=\frac{(z/2)^\nu}{\Gamma(\frac 1 2)\Gamma(\nu+\frac1 2)} \int_{-1}^1 e^{izt} (1-t^2)^{\nu-\frac 1 2} dt.
\end{equation}
To verify equivalence of the two definitions, one can integrate by parts to check that the right-hand side of identity \eqref{alt_Bessel} satisfies both recurrence relations \eqref{Recurrence1} and \eqref{Recurrence2}, and then appeal to a uniqueness result for ordinary differential equations. Any of the two definitions can be used to check the following uniform estimate, valid for every $\nu\geq 0$ and $x\in\R$:
$$|J_\nu(x)|\leq 1.$$
We will need to know the value of some finite integrals involving Bessel functions.
\begin{lemma}\label{BesselComp}
Let $\nu,\rho>0$. Then:
$$\int_0^\rho J_{\nu-1}(r)r^{\nu}dr=J_{\nu}(\rho) \rho^{\nu}.$$
\end{lemma}
\begin{proof}
Use the series representation \eqref{def_Bessel} for the function $J_{\nu-1}$ and integrate term by term. This is allowed in view of the uniform convergence of the series and the compactness of $[0,\rho]$.
\end{proof}

\noindent Another classical observation is the following: maxima and minima of Bessel functions along the positive half-line $\R_+:=\{x\in\R: x> 0\}$ steadily decrease in absolute value as $|x|$ increases. 
\begin{lemma}\label{BesselExt}
For $\nu>0$, 
let $\{\theta_k^{\nu}\}$ be the ordered sequence of stationary points of the function $J_\nu$ on the positive half-line, i.e., $0<\theta_0^{\nu}<\theta_1^{\nu}<\theta_2^{\nu}<\ldots$ and $J'_\nu(\theta_k^\nu)=0$ for every $k\in\N$. 
Then the sequence $\{|J_\nu(\theta_k^{\nu})|\}$ is monotonically decreasing in $k$.
\end{lemma}
\begin{proof}
We start by arguing as in \cite[p.~485--486]{W} to see that $\theta_0^{\nu}\geq\nu$. From the power series \eqref{def_Bessel} for $J_\nu(x)$ and the corresponding one for $J'_\nu(x)$ it is obvious that these functions are positive for sufficiently small values of $x>0$. Equation \eqref{BesselODE} can be rewritten as
$$x\frac{d}{dx}\Big(xJ'_\nu(x)\Big)=(\nu^2-x^2)J_\nu(x),$$
from which one sees that, as long as $x<\nu$ and $J_\nu(x)$ is positive, the function $x J'_\nu(x)$ is positive and increasing. It follows that $\theta_0^{\nu}$ cannot be less than $\nu$, as claimed.
Let us now consider the following auxiliary function:
$$M(x):=J_\nu^2(x)+\frac{x^2 J_\nu'(x)^2}{x^2-\nu^2}.$$
The differential equation \eqref{BesselODE} implies that 

$$M'(x)=-2x^3\Big(\frac{ J'_\nu(x)}{x^2-\nu^2}\Big)^2<0\textrm{ for every }x\geq\nu.$$
  Since we already established the lower bound $\theta_0^\nu\geq \nu$, it follows that 
the sequence $\{M(\theta_k^{\nu})\}$ decreases monotonically as $k$ increases. But $M(\theta_k^{\nu})=J^2_\nu(\theta_k^{\nu})$, and so the same holds for the sequence $\{|J_\nu(\theta_k^{\nu})|\}$.
\end{proof}

\subsection{Integrals over spheres}
Let $(\mathbb{S}^{d-1},\sigma_{d-1})$ denote the $(d-1)$-dimensional unit sphere equipped with the standard surface measure $\sigma_{d-1}$. We omit the subscript on $\sigma_{d-1}$ when clear from the context, and denote the total surface measure of the unit sphere by
\begin{equation}\label{area_sphere}
\omega_{d-1}:=\sigma\big(\mathbb{S}^{d-1}\big) = \frac{2\,\pi^{d/2}}{\Gamma(d/2)}.
\end{equation}
In polar coordinates, a measurable function $f:\R^d\rightarrow\R$ can be integrated as follows:
\begin{equation}\label{polar_integration}
\int_{\R^d} f(x)dx=\int_0^\infty \Big(\int_{\mathbb{S}^{d-1}} f(rx)d\sigma(x)\Big) r^{d-1}dr.
\end{equation}
In the case of a radial function $f(x)=f(|x|)$, this  boils down to
$$\int_{\R^d} f(x)dx=\omega_{d-1}\int_0^\infty f(r) r^{d-1}dr.$$
The following formula can be found in \cite[Lemma A.5.2]{DX} and allows for integration of radial  functions on the sphere, i.e., functions which depend only on the inner product with a fixed direction $x\in\R^d$.
\begin{equation}\label{innerproduct_int}
\int_{\mathbb{S}^{d-1}} f(x\cdot v) d\sigma(v)=\omega_{d-2} \int_{-1}^1 f(|x|t) (1-t^2)^{\frac{d-3}{2}}dt.
\end{equation}

\subsection{Laguerre polynomials}\label{ssec:Laguerre}
For every $\nu>-1$, the Laguerre polynomials $L_n^\nu(t)$, $n=0,1,2,...$, can be defined as the orthogonal polynomials associated with the measure $d\mu_\nu(t)=t^{\nu}e^{-t}dt$, for $t>0$, up to  multiplication by a scalar. In fact, they are defined in such way that $L_n^\nu(t)$ has degree $n$, is orthogonal to $\{1,t...,t^{n-1}\}$ with respect to the measure $d\mu_\nu(t)$, and

\begin{equation}\label{high-power-laguerre}
L_n^\nu(t) = (-1)^n\frac{t^n}{n!} + \text{lower order terms}.
\end{equation}
It can be shown that
\begin{equation}\label{Laguerre_orthogonality}
\int_{0}^\infty L_n^\nu(t)L_m^\nu(t)t^{\nu}e^{-t}dt = \frac{\Gamma(n+\nu+1)}{n!}\delta(n-m).
\end{equation}
Laguerre polynomials satisfy the following asymptotic identity due to Fej\'er
\begin{equation}\label{laguerre-asymp}
x^{\nu/2+1/4}e^{-x/2}L_n^\nu(x) = \pi^{-1/2}n^{\nu/2-1/4}\cos\bigg(2\sqrt{nx}-\frac{\nu\pi}{2}-\frac{\pi}{4}\bigg) + \mathcal{O}(n^{\nu/2-3/4}),
\end{equation}
where the bound for the remainder holds uniformly for $x$ in any compact subset of $(0,\infty)$. We also have that
\begin{equation}\label{laguerre-at-zero}
L_n^\nu(0) = \binom {n+\nu} n \sim \frac{n^{\nu}}{\Gamma(\nu+1)},
\end{equation}
and the following generating function
\begin{equation}\label{laguerre-gen-func}
\sum_{n=0}^\infty t^n L_n^\nu(x) = \frac{e^{-tx/(1-t)}}{(1-t)^{\nu+1}},
\end{equation}
where the limit is uniform for $x$ in any compact set of $(0,\infty)$, for fixed $t\in(-1,1)$.
It is well-known that Laguerre polynomials form an orthogonal basis of the space $L^2(\R_+,d\mu_\nu)$. In other words, if $f:\R_+\to \mathbb{C}$ is a measurable function such that
$$
\int_{0}^\infty |f(t)|^2t^{\nu}e^{-t}dt < \infty,
$$
then there exists a unique sequence of numbers $\{f_n\}$, such that
$$
f(t)=\sum_{n= 0}^\infty f_nL_n^\nu(t)
$$
in the $L^2(\R_+,d\mu_\nu)$ sense. Moreover, by identity \eqref{Laguerre_orthogonality}, we have
\begin{equation}\label{Laguerre_L2_norm}
\int_{0}^\infty |f(t)|^2t^{\nu}e^{-t}dt = \sum_{n= 0}^\infty |f_n|^2 \frac{\Gamma(n+\nu+1)}{n!}.
\end{equation}
All these properties can be found in \cite[Chapter 5]{Sz}, while Fej\'er's formula \eqref{laguerre-asymp} is contained in \cite[Theorem 8.22.1]{Sz}.

For the remainder of this section, let $\nu=d/2-1$, where $d$ denotes the dimension. An important property about Laguerre polynomials is the following:  

\begin{lemma}
Let $f:\R^d\to\R$ be the radial function defined by $f(x)=L_n^{\nu}(2\pi|x|^2)e^{-\pi|x|^2}$. 
Then its Fourier transform, normalized as in \eqref{FT}, is given by
\begin{equation}\label{Laguerre-multiplier}
\widehat{f}(y) = (-1)^n L_n^{\nu}(2\pi|y|^2)e^{-\pi|y|^2}.
\end{equation}
\end{lemma}
\begin{proof}
Identity \eqref{Laguerre-multiplier} can be deduced as follows.  Firstly, if $f:\R^d\to\R$ is a radial function, then $\widehat{f}$ is also radial, and using \eqref{innerproduct_int} together with \eqref{alt_Bessel}, we obtain
\begin{equation}\label{radial-fourier}
s^{\nu}\widehat{f}(s) =2\pi \int_0^\infty r^\nu f(r)J_\nu(2\pi rs)rdr, 
\end{equation}
for every $s>0$. Secondly, the identity in \cite[7.421--4, p.~812]{GR} states that
\begin{equation}\label{magical-identity}
\int_0^\infty x^{\nu+1}e^{-\beta x^2}L^\nu_n(\alpha x^2)J_\nu(xy)dx=2^{-\nu-1}\beta^{-\nu-n-1}(\beta-\alpha)^n y^\nu e^{-\frac{y^2}{4\beta}}L_n^\nu\bigg[\frac{\alpha y^2}{4\beta(\alpha-\beta)}\bigg],
\end{equation}
for every $\alpha\in\R$, $\beta>0$ and $y\in\R$. Choosing the appropriate values of $\alpha$ and $\beta$, one can easily deduce identity \eqref{Laguerre-multiplier} from \eqref{radial-fourier} and \eqref{magical-identity}.
\end{proof}
Using the orthogonality relation \eqref{Laguerre_orthogonality}, together with a suitable change of variables, one deduces that any radial, square-integrable function $f:\R^d \to \R$  can be uniquely expanded as
$$
f(x)= \sum_{n= 0}^\infty f_n L_n^{\nu}(2\pi|x|^2)e^{-\pi|x|^2},
$$
where the convergence holds in the $L^2(\R^d)$ sense. 
To conclude, let us mention that Laguerre polynomials are related to Hermite polynomials from \S \ref{ssec:Hermite} in the following way:
\begin{equation*}
H_{2m}(x)=(-1)^m2^{2m} m!L_m^{-1/2}(x^2) \ \ \ \ \text{and} \ \ \ \ H_{2m+1}(x)=(-1)^m2^{2m+1} m!xL_m^{1/2}(x^2).
\end{equation*}

\subsection{One-dimensional rearrangements}
Our discussion starts with the well-known {\it layer cake representation} \cite[\S 1.13]{LL}. Every nonnegative measurable function $f:\R\to\R$ can be written as an integral of the characteristic function of its superlevel sets,
\begin{equation}\label{LayerCake}
f(x)=\int_0^\infty \chi_{\{f>t\}}(x)dt.
\end{equation}
This formula alone already allow us to establish the following elementary inequality of rearrangement flavor which will  be important in applications.
\begin{lemma}\label{Rearr1}
Let $a<b$ and let $f,g:[a,b]\to\R$ be nonnegative, measurable, bounded functions. 
Further assume that $\|f\|_{L^\infty}\leq 1$. 
If  $g$ is nonincreasing, then
$$\int_{b-\|f\|_{L^1}}^bg(x)dx\leq \int_a^b f(x)g(x)dx\leq\int_a^{a+\|f\|_{L^1}} g(x)dx,$$
whereas the reverse inequalities hold if $g$ is nondecreasing.
\end{lemma}
\begin{proof}
We prove the upper bound under the assumption that $g$ is nonincreasing, all other cases being similar. By an appropriate change of variables, no generality is lost in assuming, as we will, that $[a,b]=[0,1]$.
Since $g$ is monotonic, it can have at most countably many discontinuities. In particular, one can redefine $g$ on a set of measure zero and assume that its superlevel sets $\{g>t\}=(0,\ell(t))$ are open intervals. By the layer cake representation and Fubini's theorem,
\begin{align*}
\int_0^1 fg
&=\int_0^1 f(x)\Big(\int_0^\infty \chi_{\{g>t\}}(x)dt\Big)dx\\
&=\int_0^\infty\Big(\int_0^1 f(x)\chi_{(0,\ell(t))}(x) dx\Big)dt\\
&=\int_0^\infty\Big(\int_0^{\ell(t)} f(x) dx\Big)dt.
\end{align*}
Since $\|f\|_{L^\infty}\leq 1$, the inner integral in this last expression is bounded by $\min\{\ell(t),\|f\|_{L^1}\}$. On the other hand,
$$\int_0^{\infty} \min\{\ell(t),\|f\|_{L^1}\}dt
=\int_0^{\infty} \Big(\int_0^{\|f\|_{L^1}} \chi_{(0,\ell(t))}(x)dx\Big)dt
=\int_0^{\|f\|_{L^1}}g(x)dx,$$
and the proof is complete.
\end{proof}
Let $A\subset \R$ be a measurable subset of the real line of  finite Lebesgue measure, $|A|<\infty$. The symmetric rearrangement of the set $A$, denoted $A^*$, is defined to be the open interval centered at the origin whose length equals $|A|$. We further define $\chi_A^*:=\chi_{A^*}$, and use formula \eqref{LayerCake} to extend this definition to generic nonnegative measurable functions. More precisely, the symmetric-decreasing rearrangement $f^*$ of a nonnegative measurable function $f:\R\to\R$ is defined as
$$f^*(x)=\int_0^\infty \chi^*_{\{f>t\}}(x)dt.$$
Thus $f^*$ is a lower semicontinuous function. The functions $f$ and $f^*$ are equimeasurable, i.e.,
$$|\{x\in\R: f(x)>t\}|=|\{x\in\R: f^*(x)>t\}|$$
for every $t>0$. In particular,
$$\|f\|_{L^p(\R)}=\|f^*\|_{L^p(\R)}$$
for all $1\leq p\leq \infty.$ Further note that symmetric-decreasing rearrangements are order preserving:
$$f\leq g \Rightarrow f^*\leq g^*.$$
This follows immediately from the fact that the inequality $f(x)\leq g(x)$ for all $x$ is equivalent to the statement that the superlevel sets of $g$ contain the superlevel sets of $f$.
One of the simplest rearrangement inequality for functions goes back to Hardy and Littlewood \cite[Theorem 378]{HLP} and can be informally phrased as follows. If $f,g$ are nonnegative functions on $\R$ which vanish at infinity, then
\begin{equation}\label{HL}
\int_{-\infty}^\infty f(x)g(x)dx\leq \int_{-\infty}^\infty f^*(x)g^*(x)dx,
\end{equation}
with the understanding that when the left-hand side is infinite so is the right-hand side.
This can be used in conjunction with the previous lemma to establish the following simple but useful result where, in contrast to Lemma \ref{Rearr1}, no monotonicity assumption is imposed on the function $g$.
\begin{lemma}\label{Rearr2}
Let $a<b$ and let $f,g:[a,b]\to\R$ be nonnegative, measurable, bounded functions. Further assume that $\|f\|_{L^\infty}\leq 1$. Then
$$\inf_{|J|=\|f\|_{L^1}} \int_J g\leq\int_{[a,b]} fg\leq \sup_{|J|=\|f\|_{L^1}}\int_J g$$
where infimum and supremum are taken over all measurable subsets of $[a,b]$ with measure $\|f\|_{L^1}$.
 \end{lemma}
\begin{proof}
We start by establishing the upper bound, and set $\theta:=\|f\|_{L^1}$. 
Again assume that $[a,b]=[0,1]$.
Using Hardy-Littlewood's inequality \eqref{HL} and Lemma \ref{Rearr1}, we have that
$$\int_{0}^{1} fg
\leq \int_{-\frac{1}{2}}^{\frac{1}{2}} f^* g^*
= 2\int_{0}^{\frac{1}{2}} f^* g^*
\leq 2\int_0^{\frac{\theta}{2}} g^*(x)dx
=\int_{-\frac{\theta}{2}}^{\frac{\theta}{2}} g^*(x)dx.$$
The layer cake representation and the equimeasurability of $g$ and $g^*$ then imply that 
$$\int_{-\frac{\theta}{2}}^{\frac{\theta}{2}} g^*(x)dx=\int_J g,$$
where $J$ is any measurable subset of $\{g>g^*(\theta/2)\}$ 
satisfying $|J|=\theta$ and such that $J\supseteq \{g>\lambda\}$ for every $\lambda>g^*(\theta/2)$. The result follows.
For the lower bound, one repeats the argument with the function $1-f$ instead of $f$.
\end{proof}

\section{Preliminary reductions}\label{sec:Reductions}

Theorems \ref{1Dthm} and \ref{Dthm} are phrased in terms of nonzero, radial, real-valued, integrable functions $f:\R^d\to\R$ with an integrable Fourier transform $\widehat{f}$  such that $f(0)\leq 0$ and $\widehat{f}(0)\leq 0$.
The purpose of this chapter is to describe several arguments from \cite{BCK} which reduce the problem to a more tractable class of functions.

\subsection{A trivial reduction}
We lose no generality in assuming, as we will, that the function $f$ is normalized in $L^1$:
$$\|f\|_{L^1(\R^d)}=1.$$

\subsection{Reduction to radial functions}
In the one-dimensional situation, a function is radial if and only if it is even. In higher dimensions, it turns out that one can still restrict attention to radial functions. 
To see why this is the case, start by defining $f^\sharp(x)$ to be the invariant integral of $f$ over the sphere of radius $|x|$:
$$f^\sharp(x)
:=\frac 1 {\omega_{d-1}}\int_{\mathbb{S}^{d-1}} f(|x|v)d\sigma(v).$$
This defines a radial function which satisfies  $\widehat{(f^\sharp)}=(\widehat{f})^\sharp$. 
To check this claim,
let $\mu$ be the normalized Haar measure on the compact rotation group $SO(d)$, consisting of $d\times d$ orthogonal matrices of determinant 1. Since $\mu(SO(d))=1$ and the spherical measure $\sigma$ is invariant under the action of $SO(d)$, Fubini's theorem and a change of variables imply that
\begin{align*}
f^\sharp(x)
&=\frac 1 {\omega_{d-1}}\int_{SO(d)}\Big(\int_{\mathbb{S}^{d-1}} f(|x|v)d\sigma(v)\Big) d\mu(\rho)\\
&=\frac 1 {\omega_{d-1}}\int_{\mathbb{S}^{d-1}}\Big(\int_{SO(d)} (f\circ \rho)(|x|v)d\mu(\rho)\Big)d\sigma(v)\\
&=\int_{SO(d)} (f\circ \rho)(x)d\mu(\rho).
\end{align*}
For any rotation $\rho\in SO(d)$, $\widehat{f\circ \rho}=\widehat{f}\circ \rho$. The claim follows, for then
$$\widehat{(f^\sharp)}(y)=\int_{SO(d)} \widehat{f\circ \rho}(y)d\mu(\rho)=\int_{SO(d)} (\widehat{f}\circ\rho)(y)d\mu(\rho)=(\widehat{f})^\sharp(y).$$
Moreover, it is not difficult to see that the functions $f^\sharp$ and $\widehat{f}^\sharp$ are not identically zero as long as $A(f)<\infty$ and $A(\widehat{f})<\infty$. 
{By considering the set $\{|x| > A(f)\}$, one sees that the only way for $f^\sharp$ to vanish identically in that set is if $f$ is compactly supported.
Then Schwartz's Paley-Wiener theorem \cite{Sch} implies that the function $\widehat{f}$ is analytic provided $A(f)<\infty$. But $f^\sharp=0$ also implies that $(\widehat{f})^\sharp=\widehat{(f^\sharp)}=0$, and so
$$\textrm{supp}(\widehat{f})\subset \{|y|\leq A(\widehat{f})\}$$ 
which contradicts the analyticity of $\widehat{f}$ unless $A(\widehat{f})=\infty$.}
Finally, one observes that $A(f^\sharp)\leq A(f)$ and $A(\widehat{f}^\sharp)\leq A(\widehat{f})$. 
It follows that one can restrict attention to radial functions, as claimed.

\subsection{Reduction to $f = \widehat{f}$} We lose no generality in assuming that
$$A(f) =A(\widehat{f}),$$
for otherwise we can apply a dilation $f(x) \mapsto f(x/\lambda)$ for some $\lambda > 0$. In the one-dimensional situation, this acts on the Fourier side as $\widehat{f}(y) \mapsto \lambda \widehat{f}( \lambda y)$, and
therefore does not change the product of these two quantities. However, once these two terms coincide, we can define
$$g := f + \widehat{f},$$
and it is easy to see that $A(g) \leq A(f)$.
Since $\widehat{g} = g$, it thus suffices to consider functions which equal their Fourier transform. 
In higher dimensions, we first appeal to the reduction to radial functions established above, and then the same dilation argument applies.

\subsection{Reduction to $f(0) = 0$} Following the reasoning above, suppose that $\widehat{f} = f$. Since $\widehat {e^{-\pi |\cdot|^2}} = e^{-\pi |\cdot|^2}$ in all dimensions, we can instead consider the function
$$ g := f - f(0) e^{-\pi |\cdot|^2}$$
 whenever $f(0) < 0$.
Clearly, the function $g$ coincides with its Fourier transform, satisfies $g(0)=0$, and furthermore
$$A(g)<A(f)$$
because the Gaussian always takes positive values.\\

\subsection{Square-integrability} 
Since $f$ is radial, and assuming as we may that $f=\widehat{f}$, we see that 
$$f(x)=\int_{\R^d}{f(y) \cos{(2\pi x\cdot y )}  dy}, \quad \mbox{and thus} \quad |f(x)|  \leq   \|f\|_{L^1(\mathbb{R}^d)}.$$
Taking the supremum in $x$ yields
$$ \|f\|_{L^\infty(\R^d)} \leq \|f\|_{L^1(\mathbb{R}^d)},$$
and therefore
$$ \|f\|_{L^2(\mathbb{R}^d)}
\leq \|f\|_{L^\infty(\R^d)}^{1/2} \| f\|_{L^1(\mathbb{R}^d)}^{1/2} 
\leq \|f\|_{L^1(\mathbb{R}^d)}
< \infty.$$
Therefore, we lose no generality in assuming that $f$ is square-integrable.  Note that, for the type of functions we are interested in, the $L^1$ and $L^2$ norms will always be comparable. For instance, if $d=1$, then
$$ \frac{\|f\|_{L^1(\mathbb{R})}}{2} \leq  \int_{-A(f)}^{A(f)}{|f(x)|dx} \leq \sqrt{2 A(f)}\left(\int_{-A(f)}^{A(f)}{|f(x)|^2dx}\right)^{\frac{1}{2}} \leq \sqrt{2 A(f)}\|f\|_{L^2(\mathbb{R})},$$
and we care about functions $f$ for which $A(f)$ is as small as possible.
}

\subsection{An easy lower bound}
The previous reductions allow us to restrict attention to functions $f:\R^d\to\R$ which satisfy the following set of assumptions.
\begin{align}
&f\in L^1(\R^d)\cap L^2(\R^d): \|f\|_{L^1(\R^d)}=1,\label{nonzero_integrable}\\
&f \textrm{ is real-valued,}\label{even_Rvalued}\\
&f(0)=0,\label{zero_at_zero}\\
&f=\widehat{f},\label{Fourier}\\
& f\textrm{ is radial.}\label{Radial}
\end{align}
Observe that functions $f$ which satisfy  assumptions \eqref{nonzero_integrable} and \eqref{Fourier} are uniformly continuous and bounded with $\|f\|_{L^\infty}\leq 1$. Moreover, in view of the Riemann-Lebesgue lemma,
$$\lim_{|x|\rightarrow\infty} |f(x)|=0.$$
Functions satisfying \eqref{Fourier} cannot be compactly supported unless they are identically zero. Moreover, assumptions \eqref{zero_at_zero} and \eqref{Fourier} imply  
$$\int_{\R^d} f(x)dx=\widehat{f}(0)=0.$$
The following simple argument from \cite{BCK} establishes \textit{some}
 lower bound for $A(f)$.

 \begin{lemma}\label{trivialLB}
Let $f:\R\to\R$ be a function satisfying assumptions   \eqref{nonzero_integrable}--\eqref{Fourier}. Then
$$A(f) \geq \frac{1}{4}.$$
\end{lemma}
\begin{proof}
Since $\|f\|_{L^1}=1$ and $f$ has zero average, 
it follows that
\begin{equation}\label{L1PosNeg}
\int_{\{f> 0\}} f^+(x)dx=\int_{\{f<0\}} f^-(x)dx=\frac1 2,
\end{equation}
where $f^+$ and $f^-$ denote the positive and negative part of the function $f$, respectively. Consequently,
$$ \frac{1}{2} = \int_{\{f< 0\}}{f^-(x)dx}= \int_{\{f< 0\}}{|f(x)| dx} \leq \int_{\{f < 0\}}{1~dx} = |\left\{x\in\R:f(x) < 0\right\}|.$$
By definition of $A(f)$, we have $\{f < 0\}\subseteq [-A(f),A(f)]$, and this implies the desired bound.
\end{proof}

 {\it Remark.} This argument carried out in higher dimensions leads to the lower bound given by Theorem \ref{BCKhigherD}.

\section{Proof of Theorem \ref{1Dthm}}\label{sec:1DThm}
In this chapter, we prove Theorem \ref{1Dthm}. We first establish the lower bound $A(f) \geq 0.45$. 
With some additional work, our argument can be refined to  yield $A(f) \geq 0.453$. 
However,  we do not believe that  lower bound to be close to best possible, and so we opted for clarity of exposition over a sharper form. The upper bound $\inf_f A(f)\leq 0.594$ follows from an explicit construction described in \S \ref{sec:1DUB} below.

\subsection{Proof of the lower bound} 
Let $f:\R\to\R$ be a function satisfying assumptions \eqref{nonzero_integrable}$-$\eqref{Fourier}, which throughout this section we simply refer to as  an {\it admissible function}. 
Since $f$ is an even function, it is enough to study  its behavior on the positive half-line. The argument is based on understanding the size of the quantity
$$\int_{0}^{A}{f(x)dx}$$
for $A:=A(f)$.
 This integral accounts for half of the negative mass, which equals $-1/4$ since $\|f\|_{L^1}=1$ and $\int f=0$, but might also contain some of
the positive mass. We will derive a pointwise upper bound for the function 
$f$ which places fairly strong restrictions on its positive part $f^+$ inside the interval $[0, A]$. 
As a consequence,
$$\mbox{if}\;\tau:= \int_{0}^{A}{f^{+}(x)dx} \; \mbox{were large, then} \; |\left\{x \in [0,A]: f(x) > 0\right\}|\; \mbox{would have to be large.}$$
On the other hand, from $\|f\|_{L^\infty} \leq 1$ one infers that
$$ |\left\{x \in [0,A]: f(x) \leq 0\right\}| \geq \frac{1}{4}, \; \mbox{and this implies} \;   |\left\{x \in [0,A]: f(x) > 0\right\}| \leq A-1/4.$$
We will use this to show that if $A < 0.45$, then
\begin{equation}\label{tauUB}
\tau < \frac{13}{500}.
\end{equation}
The final ingredient is an explicit integral identity derived from $f=\widehat{f}$ which will be used to perform a bootstrap-type argument that yields a contradiction. We now turn to the details.

\begin{lemma}\label{PointwiseUB}
Let $f$ be an admissible function, and set $A=A(f)$.
 If $A \leq 1/2$, then for all $0 \leq x \leq A$
\begin{equation}\label{PUBeq}
 f(x) \leq \frac{1}{2} + \frac{\sin{(2\pi(A-1/4) x)} - \sin{(2 \pi A  x)}}{\pi x}.
 \end{equation}
\end{lemma}

\begin{proof}  Since $f=\widehat{f}$ and $f$ is even, we have that
\begin{align*} f(x) &= \int_{-\infty}^\infty{f(y) \cos{(2\pi x y)}  dy} \\
&= \int_{-\infty}^\infty{(f^+(y)-f^{-}(y)) \cos{(2\pi x y)} dy} \\
&\leq \frac{1}{2} - \int_{-\infty}^\infty{f^{-}(y) \cos{(2\pi x y)} dy},
\end{align*}
where in the last inequality we used the observation from \eqref{L1PosNeg} that $\|f^+\|_{L^1}=1/2$.
If $A \leq 1/2$, then the function $y\mapsto\cos(2 \pi x y)$ is nonnegative and monotonically decreasing on $[0,A]$  for every $0 \leq x \leq A$. Since $f^-$ is even and $\|f^{-}\|_{L^\infty} \leq 1$, it follows from
Lemma \ref{Rearr1} and an explicit computation that
$$ \int_{-\infty}^\infty{f^{-}(y) \cos{(2\pi x y)}  dy} \geq 2\int_{A-\frac{1}{4}}^{A}{\cos{(2\pi x y)} dy} =  -  \frac{\sin{(2\pi(A-1/4) x)} - \sin{(2 \pi A  x)}}{\pi x}.$$
\end{proof}

\noindent The pointwise upper bound given by Lemma \ref{PointwiseUB} can be used to establish the next ingredient.
\begin{lemma}\label{IntegralUB} 
Let $f$ be an admissible function, and set $A=A(f)$.
If $A \leq 1/2$, then
$$ \int_{0}^{A}{f^{+}(x)dx} \leq \int_{\frac{1}{4}}^{A}{\frac{1}{2} + \frac{\sin{(2\pi(A-1/4) x)} - \sin{(2 \pi A x)}}{\pi x} dx}$$
\end{lemma}
\begin{proof}
As observed before,
$|\left\{x\in [0, A]: f(x) > 0\right\}|\leq A-1/4.$
Therefore

 $$ \int_{0}^{A}{f^{+}(x)dx} = \sup_{J \subset [0,A] \atop |J| = A - 1/4}  \int_{J} f^{+}.$$
Since the pointwise upper bound given by Lemma \ref{PointwiseUB} is always nonnegative, inequality \eqref{PUBeq} remains valid if $f$ is replaced by $f^+=\max\{f,0\}$. Thus
  \begin{align*}
\sup_{J \subset [0,A] \atop |J| = A - 1/4}  \int_{J} f^{+}(x)dx
   &\leq \sup_{J \subset [0,A] \atop |J|=A - 1/4} \int_J \frac{1}{2} + \frac{\sin{(2\pi(A-1/4) x)} - \sin{(2\pi A x)}}{\pi x} dx\\
   &= \int_{\frac1 4}^A \frac{1}{2} + \frac{\sin{(2\pi(A-1/4) x)} - \sin{(2 \pi A  x)}}{\pi x} dx,
   \end{align*}
   where the last identity follows at once from noting that the function 
   $$x\mapsto \frac{1}{2} + \frac{\sin{(2\pi(A-1/4) x)} - \sin{(2 \pi A x)}}{\pi x}$$
   is nondecreasing on $[0,A]$.
\end{proof}

\noindent Lemma \ref{IntegralUB}  implies the announced upper bound \eqref{tauUB} for $\tau$. {A simple computation shows that the function
$$A \mapsto \int_{\frac{1}{4}}^{A}{\frac{1}{2} + \frac{\sin{(2\pi(A-1/4) x)} - \sin{(2 \pi A x)}}{\pi x} dx}$$
is monotonically increasing for $0.25 \leq A \leq 0.5$. In particular, if $A < 0.45$, then

\begin{equation}\label{UBtau}
\tau= \int_{0}^{A}{f^{+}(x)dx} 
\leq \int_{\frac{1}{4}}^{\frac{45}{100}}{\frac{1}{2} + \frac{\sin{(2\pi(\frac{45}{100}-1/4) x)} - \sin{(2 \pi\frac{45}{100} x)}}{\pi x} dx}
< \frac{13}{500}.
\end{equation}
We proceed to derive the relevant integral identity.
\begin{lemma}\label{IntId}
 Let $f$ be an admissible function, and set $A=A(f)$. Then
\begin{equation}\label{IntegralID}
 \int_{0}^{A}{f(x)dx} = \int_{-\infty}^\infty{f(y) \left(\frac{\sin{(2\pi A y)}}{2 \pi y} + \frac{13}{400} (8 \pi y^2 - 2)e^{-\pi y^2} \right)dy}.
 \end{equation}
\end{lemma}

\noindent {\it Remark.} The factor $13/400$ in identity \eqref{IntegralID}  may seem peculiar. While the identity  remains valid if $13/400$ is replaced by any other real number,
this particular choice turns out to be essentially optimal with respect to subsequent arguments.
\begin{proof} 
The proof proceeds in two steps.
The first step starts similarly to the proof of Lemma \ref{PointwiseUB}, and via Fubini's theorem and an explicit integration yields
\begin{align*}
\int_{0}^{A}{f(x)dx} &= \int_{0}^{A}{\left( \int_{-\infty}^\infty{f(y) \cos{(2 \pi x y)} dy} \right)dx}\\
&= \int_{-\infty}^\infty{ f(y)  \left(\int_{0}^{A}{\cos{(2 \pi x y)} dx}\right)  dy} \\
&= \int_{-\infty}^\infty{f(y)\frac{\sin{(2\pi A y)}}{2 \pi y} dy}
\end{align*}
The second step uses the fact that a square-integrable function satisfying $f = \widehat{f}$ admits an Hermite expansion of the form \eqref{HermiteExpansion}, where only Hermite functions $\psi_{4n}$ whose degree is divisible by 4 appear with nonzero coefficients. Since Hermite functions are mutually orthogonal as quantified by \eqref{HermiteOrthogonality},
any  function $\psi_{4n}$ is orthogonal to $\psi_2(y)=2^{-5/4}(8\pi y^2-2)e^{-\pi y^2}$, and therefore so is $f$. 
\end{proof}

\begin{proof}[Proof of the lower bound $A(f) \geq 0.45$] 
As usual, let $f$ be an admissible function and set $A:=A(f)$. Also, recall the auxiliary function from Lemma \ref{IntId} which we now denote by

$$\Upsilon_A(x):=\frac{\sin{(2\pi A x)}}{2 \pi x} + \frac{13}{400} (8 \pi x^2 - 2)e^{-\pi x^2}.$$
By definition of $\tau$ and identity \eqref{IntegralID}, we have that
\begin{align}
-\frac{1}{4} + \tau =  \int_{0}^{A}{f(x)dx} &=  \int_{-\infty}^\infty{(f^{+}(y) - f^{-}(y))\Upsilon_A(y)dy} \notag\\
 &\geq 
\inf_{I_1 \subset [-A,A] \atop |I_1| = 2\tau} \int_{I_1}{\Upsilon_A}
 +  \inf_{I_2 \subset \mathbb{R} \setminus[-A,A] \atop |I_2| = 1/2-2\tau} \int_{I_2}{\Upsilon_A}
 - \sup_{I_3 \subset [-A,A] \atop |I_3| = 1/2} \int_{I_3}{\Upsilon_A} ,\label{ineq}
\end{align}
where the inequality results from successive applications of Lemma \ref{Rearr2}.
In greater detail: the first and the second summands on the right-hand side of \eqref{ineq} arise as lower bounds given by Lemma \ref{Rearr2} applied to the function $f^+$ on  $[-A,A]$ and $\R\setminus [-A,A]$, respectively. 
The third summand arises as (the negative of) the upper bound given by Lemma \ref{Rearr2}  applied to the function $f^-$ on $[-A,A]$.

The rest of the proof proceeds by contradiction.
From \eqref{UBtau} we know that $A < 0.45$ implies $0 \leq \tau < 13/500$, and so the result will follow once we show that  inequality \eqref{ineq} fails  for every $\tau$ in this range.
To establish this fact, it suffices to establish failure at the endpoint $\tau = 13/500$. To see why this is the case, start by noting that the third summand on the right-hand side of inequality \eqref{ineq} does not depend on the parameter $\tau$. 
It suffices to study the functions
\begin{equation}\label{h1h2}
 h_1(\tau) :=\inf_{I_1 \subset [-A,A] \atop |I_1| = 2\tau} \int_{I_1}{\Upsilon_A}
 \;\;\;\textrm{ and }\;\;\;
 h_2(\tau) :=  \inf_{I_2 \subset \mathbb{R} \setminus[-A,A] \atop |I_2| = 1/2-2\tau} \int_{I_2}{\Upsilon_A}.
\end{equation}
The plan is the following: if inequality \eqref{ineq} holds for some $\tau_0>0$, then we show that it also holds for every larger $\tau > \tau_0$. 
This in turn follows from the fact that, on the interval $\tau \in [0,13/500)$,
\begin{equation}\label{LipGoal}
h:= h_1 + h_2 \; \mbox{is a Lipschitz function of $\tau$ with Lipschitz constant Lip}(h) < 1.
\end{equation}
An explicit computation shows that
inequality \eqref{ineq} fails at the endpoint $\tau = 13/500$ for any $A<0.45$, and this yields the desired contradiction. 
It remains to prove assertion \eqref{LipGoal}.
We start by noting an alternative representation for the functions $h_1, h_2$ which is based on identifying the optimal sets in the expressions \eqref{h1h2}. 
The infimum is actually a minimum, and the optimal set $I_1^*=I_1^*(\tau,A)$ for $h_1$ is given by
\begin{equation}\label{I1star}
 I_1^* := \left\{x \in [-A,A]:  \Upsilon_A(x) \leq c_1\right\},
 \end{equation}
where the parameter $c_1=c_1(\tau,A)$ is uniquely determined by
$$ c_1 =  \inf \left\{y \in \mathbb{R}: |\left\{x \in [-A,A] : \Upsilon_A(x) \leq y\right\}| \geq 2\tau \right\}.$$
 In a similar way, the optimal set $I_2^*=I_2^*(\tau,A)$ for the function $h_2$
is given by
\begin{equation}\label{I2star}
I_2^* := \left\{x \in \mathbb{R} \setminus [-A,A]:  \Upsilon_A(x) \leq c_2\right\},
 \end{equation}
where 
$$c_2=\inf \left\{y \in \mathbb{R}: |\left\{x \in \mathbb{R} \setminus [-A,A]: \Upsilon_A(x) \leq y\right\}| \geq \frac{1}{2} - 2\tau \right\}.$$
In other words,
\begin{equation}\label{h1Rep}
 h_1(\tau) = \int_{I_1^*}{\Upsilon_A}\;\textrm{ and } \;h_2(\tau) = \int_{I_2^*}{\Upsilon_A},
 \end{equation}
 where the sets $I_1^*$ and $I_2^*$ are respectively given by \eqref{I1star} and \eqref{I2star}; see also Figure \ref{fig:intervals}.
It is  straightforward to check that  $h_1$ and $h_2$ are nondecreasing functions of $\tau$.   
As we will see, $h_1$ and $h_2$ are actually differentiable functions of $\tau$. 
For the type of Lipschitz bounds which we seek to establish,
the following rough estimates suffice: for $y\geq 0$ and $A < 0.45$,
\begin{equation}\label{gAEst}
 \Upsilon_A(y) \leq 0.39  \mbox{ if}~ y\in\left[0, \frac{1}{10}\right], 
 \;\;\;\textrm{ and }\;\;\;
\Upsilon_A(y)\geq -0.09  \mbox{ if}~y\notin\left[\frac{7}{5}, \frac{9}{5}\right].
\end{equation}
As $\tau$ increases, $h_2(\tau)$ computes the integral over a smaller area of the most negative part of the function $\Upsilon_A$. The second
bound in \eqref{gAEst} implies that, for 
$$ \frac{1}{2} - 2\tau \geq \frac{9}{5} - \frac{7}{5} \Longleftrightarrow \tau \leq \frac{1}{20},$$
the optimal set $I_2^*(\tau)$ will get smaller in a region where the function $\Upsilon_A$ is, albeit negative, larger than $-0.09$. 
Let $0\leq\tau_0\leq 1/20$. For sufficiently small $\varepsilon>0$, we have that $I_2^*(\tau_0+\varepsilon)\subset I_2^*(\tau_0)$.
Since 
$$|I_2^*(\tau_0+\varepsilon)|=\frac12-2(\tau_0+\varepsilon) \quad \mbox{and} \quad |I_2^*(\tau_0)|=\frac12-2\tau_0,$$ 
we see that the set $K:=I_2^*(\tau_0)\setminus I_2^*(\tau_0+\varepsilon)$ has measure $|K|=2\varepsilon$. By H\"older's inequality, it then follows that 
\begin{equation}\label{Holder}
h_2(\tau_0+\varepsilon)-h_2(\tau_0)
=\int_{K}\Upsilon_A 
\leq   \|\Upsilon_A\|_{L^\infty(K)}  \cdot |K|
\leq 0.09\cdot 2\varepsilon.
\end{equation}
Dividing the left and right most sides of this chain of inequalities by $\varepsilon$, and letting $\varepsilon\to 0^+$, yields
$$ \frac{d h_2}{d\tau} (\tau) \leq 2 \cdot 0.09 = 0.18 \;\;\; \mbox{for} \;\;\; \tau \leq \frac{1}{20}.$$
In a similar but slightly simpler way, using instead the first bound in \eqref{gAEst}, one can verify that
$$ \frac{d h_1}{d\tau}(\tau) \leq 2 \cdot 0.39 = 0.78 \;\;\; \mbox{for} \;\;\; \tau \leq \frac{1}{10}. $$
As a consequence, Lip$(h_1+h_2)\leq 0.96<1$ on the interval $\tau\in[0,1/20]\supset [0,13/500)$.
This establishes \eqref{LipGoal} and completes the proof of Theorem \ref{1Dthm} except for the upper bound which is the subject of the next section.
\end{proof}

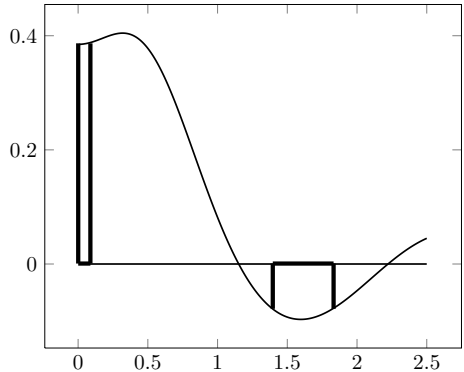
\begin{figure}[h!]
\centering
\begin{tikzpicture}[scale=0.8]
\begin{axis}
\addplot[samples=500,domain=0.01:2.5, thick]{sin(2*pi*0.45*deg(x))/(2*pi*x) + 13/(400)*(8*pi*x^2-2)*exp(-pi*x*x)};
\addplot[samples=500,domain=0.01:2.5, thick]{0};
\end{axis}
\draw [ultra thick] (3.75,1.4) -- (4.75,1.4);
\draw [ultra thick] (3.75,1.4) -- (3.75,0.65);
\draw [ultra thick] (4.75,1.4) -- (4.75,0.65);

\draw [ultra thick] (0.55,1.4) -- (0.75,1.4);
\draw [ultra thick] (0.55,1.4) -- (0.55,5.05);
\draw [ultra thick] (0.75,1.4) -- (0.75,5.05);
\end{tikzpicture}
\captionsetup{width=0.95\textwidth}
\caption{Intervals $I_1^*$ (on the left) and $I_2^*$ (on the right) for $\Upsilon_{A}$ at $A=0.45$ and $\tau \sim 0.02$.}
\label{fig:intervals}
\end{figure}

\subsection{Proof of the upper bound by an explicit example}\label{sec:1DUB}
 This short section follows \cite[\S 2]{BCK} in spirit. As noted in \S \ref{ssec:Hermite}, any linear combination of suitably rescaled Hermite functions 
$$ f(x) = \sum_{n=0}^{\infty}{\alpha_n H_{4n}(\sqrt{2\pi} x) e^{-\pi x^2}}$$
satisfies $f=\widehat{f}$. A straightforward method to construct functions which satisfy assumptions \eqref{nonzero_integrable}$-$\eqref{Fourier} consists in simply
choosing finitely many nonzero coefficients $\{\alpha_n\}$ in such a way that $f(0) = 0$. By direct search (more precisely, by a greedy-type algorithm where
previously found candidates are perturbed in a favorable direction by adding a new function),
 we found the example
$$ \alpha_0 = -\frac{113}{100} \qquad \alpha_1 = \frac{1}{25} \qquad \alpha_2 = \frac{1}{3240} \qquad \alpha_3 = \frac{-\alpha_0 - 12\alpha_1 - 1680 \alpha_2}{665280} \qquad \alpha_n=0 \textrm{ if } n\geq 4$$
The arising function satisfies all assumption of Theorem \ref{1Dthm}, has its largest root at $\sim 0.59354$ and almost a double root
at $\sim 0.8990$, and is depicted in Figure \ref{fig:goodcandidate}.
This concludes the proof of Theorem \ref{1Dthm}.\\

\textit{Remark.} Theorem \ref{ext} is implicitly constructive in the sense that it guarantees that we could improve
this upper bound by adding further Hermite functions (since it implies that no finite linear combination of Hermite functions can be an extremizer). However,
the actual numerical improvement  observed after adding a multiple of $H_{16}$ is miniscule. This leads us to believe that our candidate function is close to optimal.

\qedhere

\section{Linear flows on the torus, and consequences}\label{sec:flows}

We start by proving an elementary statement about linear flows on the torus $\mathbb{T}^d=\R^d/(2\pi\Z)^d$, stating that all of them return to a small neighborhood of the origin infinitely many times. This is not a difficult result, and stronger results are available in the literature (see e.g. \cite{katok}).
Since this weaker statement is enough for our subsequent purposes and has a very short proof, we include it here. 

\begin{lemma}\label{closeto0} 
Let $\mathbb{T}^d$ denote the $d$-dimensional torus, and let $\|\cdot\|$ denote the induced norm from $\R^d$. For  ${\bf a} \in \mathbb{T}^d$, consider the linear flow $\gamma:\R\to \mathbb{T}^d$ given by
$$ \gamma(t) =  t{\bf a}.$$
For any $\varepsilon > 0$, there exists an infinite sequence of times $t_1 < t_2 < \dots $ with $t_{i} \in \mathbb{N}$ such that
$$ \| \gamma(t_i) \| \leq \varepsilon.$$ 
\end{lemma}

\begin{proof} We equip the torus $\mathbb{T}^d$ with the normalized Haar measure $\mu$, and consider the translation map $T:\mathbb{T}^d \rightarrow \mathbb{T}^d$ given by
$$ Tx = x +\textbf{a}.$$
The map $T$ clearly preserves the  measure $\mu$. Let  $\varepsilon > 0$ be arbitrary, and consider the ball
$$ E = \left\{x \in \mathbb{T}^d: \|x\| \leq \frac{\varepsilon}{2} \right\}.$$
The Poincar\'{e} recurrence theorem for the discrete-time case \cite[p.~142]{katok} states that almost every point of $E$ returns to $E$ infinitely often under positive iterations by $T$. In other words, the set
$$ F := \left\{x \in E: ~\exists N \in \mathbb{N}: ~T^n(x) \notin E\; \mbox{for all} ~ n > N\right\} \; \mbox{has zero Haar measure,}$$
i.e. $\mu(F) = 0$.
Thus there  exists $x_0 \in E \setminus F$. By additivity of $T$, we have
$$ \gamma(n)  =  n \textbf{a} = - x_0 + ( x_0 + n \textbf{a}) = -x_0 + T^n(x_0).$$
This, together with the fact that $x_0\in E\setminus F$, implies that
$\|\gamma(n)\| \leq \varepsilon$ for infinitely many $n \in \mathbb{N}$.
\end{proof}
The construction used in the one-dimensional proof of Theorem \ref{ext} below will make use of the sequence of functions $\{\varphi_n\}$ defined as
\begin{equation}\label{def_varphi}
\varphi_n(x):=   \frac{1}{H_{4n+4}(0)} H_{4n+4}(\sqrt{2\pi} x)e^{-\pi x^2}  -\frac{1}{H_{4n}(0)}H_{4n}(\sqrt{2\pi}x) e^{-\pi x^2},
\end{equation}
where $H_n$ is the Hermite polynomial of degree $n$. 
We note that  
\begin{equation}\label{hotHn}
 H_n(x) = 2^n x^n +  \textrm{lower order terms},
 \end{equation}
and remark that
\begin{equation}\label{atzeroHn}
H_{4n}(0)=\frac{\Gamma(4n+1)}{\Gamma(2n+1)}.
\end{equation}
For every $n\in\N$, the function $\varphi_n$ coincides with its Fourier transform. It also satisfies $\varphi_n(0)=0$. Furthermore, identities \eqref{hotHn} and \eqref{atzeroHn} imply
\begin{equation}\label{PhiExpand}
 \varphi_n(x)= e^{-\pi x^2}(  a_{4n+4} x^{4n+4} + \textrm{lower order terms} ),  \quad \mbox{where} ~ a_{4n+4} =   2^{6n+6} \pi^{2n+2} \frac{\Gamma(2n+3)}{\Gamma(4n+5)} > 0,
 \end{equation}
and therefore $\varphi_n(x)>0$ as soon as $|x|$ is sufficiently large, depending on $n$.
We are not aware of any result of the following type and consider it to be of independent interest.

\begin{lemma}\label{herm} 
Let $\{a_1, a_2, \dots, a_k\} \subset \mathbb{R}_{+}$ be any finite subset of the positive half-line. Then there exist infinitely many $n \in \mathbb{N}$ such that
$$ \min_{1\leq j\leq k} \varphi_n(a_j) >0.$$
\end{lemma}

\begin{proof}
 Let $0 < a_1 < a_2 < \dots <a_k$ be given and fixed, and write $\textbf{a}=(a_1,a_2,\ldots,a_k)$. We are only interested in the values of the functions $\varphi_n$ at the points $a_j$, and can therefore replace
Hermite functions by a pointwise approximation given by the asymptotic expansion \eqref{HermiteAsympt}. 
Note that we are only dealing with indices that are a multiple of 4 and therefore get a simplified asymptotic expansion without phase shift
$$
\frac{1}{H_{4n}(0)}e^{-\pi x^2} H_{4n}(\sqrt{2\pi} x) 
= \cos{  \left(\sqrt{8n+1}\sqrt{2\pi}x \right)} + \frac{(\sqrt{2\pi}x)^3}{6\sqrt{8n+1}}\sin{\left(\sqrt{8n+1}\sqrt{2\pi} x\right)} + \mathcal{O}\left(\frac{1}{n}\right).
$$
This implies, again for fixed $x \in \mathbb{R}$,
\begin{align*}
\varphi_n(x) &= \cos{  \left(\sqrt{8n+9}\sqrt{2\pi}x \right)} - \cos{  \left(\sqrt{8n+1}\sqrt{2\pi}x \right)} \\
&+\frac{(\sqrt{2\pi}x)^3}{6\sqrt{8n+9}}\sin{\left(\sqrt{8n+9}\sqrt{2\pi} x\right)} - \frac{(\sqrt{2\pi}x)^3}{6\sqrt{8n+1}}\sin{\left(\sqrt{8n+1}\sqrt{2\pi} x\right)} +  \mathcal{O}\left(\frac{1}{n}\right),
\end{align*}
where the implicit constant in the error term may depend on $x$. Basic algebra yields
$$ \sqrt{8n+9} =  \sqrt{8n+1}+\frac{4}{\sqrt{8n+1}} + \mathcal{O}\left(\frac{1}{n^{3/2}}\right)$$
and therefore, by Taylor expansion,
\begin{align*}
 \cos{  \left(\sqrt{8n+9}\sqrt{2\pi}x \right)} &= \cos{  \left(\sqrt{8n+1}\sqrt{2\pi}x + \frac{4\sqrt{2\pi} x}{\sqrt{8n+1}} + \mathcal{O}\left(\frac{1}{n^{3/2}}\right) \right)} \\
&=  \cos{  \left(\sqrt{8n+1}\sqrt{2\pi}x\right)} -\sin{  \left(\sqrt{8n+1}\sqrt{2\pi}x\right)} \frac{4\sqrt{2\pi} x}{\sqrt{8n+1}}  + \mathcal{O}\left(\frac{1}{n}\right).
\end{align*}
The same type of argument yields 
$$ \frac{(\sqrt{2\pi}x)^3}{6\sqrt{8n+9}}\sin{\left(\sqrt{8n+9}\sqrt{2\pi} x\right)} - \frac{(\sqrt{2\pi}x)^3}{6\sqrt{8n+1}}\sin{\left(\sqrt{8n+1}\sqrt{2\pi} x\right)} =  \mathcal{O}\left(\frac{1}{n}\right),$$
where, as always, the implicit constant in the error term is allowed to depend on $x$ but not on $n$, and can be chosen uniformly in $x$ inside any interval of finite length. Therefore, for fixed $x \in \mathbb{R}$,
$$ \varphi_n(x) = -\sin{  \left(\sqrt{8n+1}\sqrt{2\pi}x\right)} \frac{4\sqrt{2\pi} x}{\sqrt{8n+1}}  + \mathcal{O}\left(\frac{1}{n}\right).$$
Finally, we note that
$$\sqrt{8n+1} = \sqrt{8n} + \frac{1}{2\sqrt{8n}} + \mathcal{O}\left(\frac{1}{n}\right),$$
and further simplify
$$ \varphi_n(x) = -\sin{  \left(4\sqrt{\pi n}x\right)} \frac{4\sqrt{2\pi} x}{\sqrt{8n+1}}  + \mathcal{O}\left(\frac{1}{n}\right).$$
{Because of} continuity properties of the sine function, it is sufficient to prove the existence of
infinitely many $n \in \mathbb{N}$ and of $\theta_{\textbf{a}}>0$ such that
$$  \sin{ \left( 4\sqrt{ \pi n}a_j \right)}  \leq -\frac{\theta_{\textbf{a}}}{2} < 0 \qquad \mbox{for every} \quad 1 \leq j \leq k.$$
Clearly, the truth of such a statement depends on where the sequence 
\begin{equation}\label{seqTd}
\left( 4\sqrt{ \pi n}a_1, 4\sqrt{ \pi n}a_2, \dots, 4\sqrt{\pi n}a_k\right)
\end{equation}
is located inside the torus $\mathbb{T}^k\cong [0, 2\pi]^k$.
We need to prove that infinitely many elements of this sequence lie in the subset 
$$[\pi + \delta, 2\pi - \delta]^k \subset  \mathbb{T}^k,$$ 
for a sufficiently small $\delta > 0$ that is allowed to depend on $\textbf{a}$ (and would guarantee the desired statement with $\theta_{\textbf{a}} = 2\sin{\delta}$).
Clearly, this sequence of points is contained in the ray $\gamma:\mathbb{R}_{+} \rightarrow \mathbb{T}^k,$
$$ \gamma(t) =   4\sqrt{ \pi }\left(a_1, a_2, \dots, a_k\right)t.$$
Thanks to the elementary fact
$$ \sqrt{n+1} - \sqrt{n} \leq \frac{1}{2\sqrt{n}} = o_n(1),$$
it suffices to show that the ray $\gamma(t)$ intersects the subset $[\pi+\delta,2\pi-\delta]^k$ for an increasing sequence of real numbers that tend to infinity: the sublinear growth of the
square root will then allow us to find nearby integers whose square roots are still mapped into that subset via $\gamma$.
It is well known that, depending on the diophantine properties of
$\textbf{a} = (a_1, \dots, a_k)$, the linear flow may or may not be dense in $\mathbb{T}^k$. However, $\{a_1,\ldots,a_k\}$ could be any collection of positive real numbers, and we cannot impose any sort of control on its number-theoretic properties.
A much simpler
argument suffices: According to Lemma \ref{closeto0}, \textit{any} linear flow on the torus will pass within any arbitrarily small neighborhood of the origin infinitely many times.
After leaving the origin, such a ray
will always intersect a subset  $[\varepsilon, \pi -\varepsilon]^k$ for some $\varepsilon > 0$ (see Figure \ref{fig:flow}). Clearly, the angle of the ray will determine the possible size of $\varepsilon$, but for a fixed direction $\textbf{a}\in\T^k$ such $\varepsilon$ can always be explicitly given. Set, for instance,
$$\varepsilon = \frac{1}{2}\frac{\min_{1 \leq j \leq k}{a_j}}{|\textbf{a}|},$$
and note that, for $ t=(2\left| \textbf{a} \right|)^{-1}$,
$$ t \textbf{a} = \left( \frac{a_1}{2\left| \textbf{a} \right|}, \frac{a_2}{2\left| \textbf{a} \right|}, \dots, \frac{a_k}{2\left| \textbf{a} \right|} \right).$$
Every entry of this vector is larger than $\varepsilon$ and smaller than 1/2, and therefore the vector is certainly contained in $[\varepsilon, \pi -\varepsilon]^k$. 
Setting $\delta=2\varepsilon$, this shows that infinitely many elements of the sequence \eqref{seqTd} lie in  $[\delta, \pi - \delta]^k \subset \mathbb{T}^k$. By symmetry (i.e. reversing the flow of time), the same result holds for $[\pi + \delta, 2\pi - \delta]^k \subset  \mathbb{T}^k$.
\end{proof}

\begin{center}
\begin{figure}[h!]
\centering
\begin{tikzpicture}[scale=3]
\draw[ultra thick] (0,0) -- (1,0);
\draw[ultra thick] (1,1) -- (1,0);
\draw[ultra thick] (1,1) -- (0,1);
\draw[ultra thick] (0,1) -- (0,0);
\draw[dashed, thick] (0.1,0.1) -- (0.1,0.4);
\draw[dashed, thick] (0.1,0.4) -- (0.4,0.4);
\draw[dashed, thick] (0.4,0.1) -- (0.4,0.4);
\draw[dashed, thick] (0.1,0.1) -- (0.4,0.1);
\draw (0,0.5) -- (1,0.5);
\draw (0.5,0) -- (0.5,1);
\draw[ultra thick] (0,0) circle (0.01cm);
\draw[thick,->] (0,0) -- (0.2,0.3);
\end{tikzpicture}
\captionsetup{width=0.9\textwidth}
\caption{A linear flow on $\mathbb{T}^2$ starting at the origin in a direction all of whose components are positive will always hit the square $[{\varepsilon}, \pi -{\varepsilon}]^2$ (dashed) for some ${\varepsilon} > 0$.}
\label{fig:flow}
\end{figure}
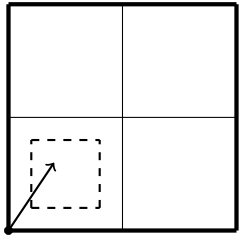
\end{center}

A closer look at the proof of Lemma \ref{herm}  suggests that in the generic case of
($a_1, a_2, \dots, a_k)$ being linearly independent over $\mathbb{Q}$  stronger results will hold: the linear flow will be uniformly
distributed, and any of the $2^k$ possible prescribed sign patterns will occur with equal frequency. However, the statement could still be true even if the entries
are not linearly independent: Linear flows on the torus, which arise as a first order limiting object, will be arbitrarily 
close to the origin infinitely often and any open neighborhood of the origin already contains all possible $2^k$ sign patterns. A more
detailed understanding could be of interest.

\subsection{Classical Hermite polynomials} Lemma \ref{herm} is a statement about a certain linear combination of Hermite functions. We now prove the corresponding result for classical Hermite polynomials, Theorem \ref{Hermite}. The proof is actually simpler than that of Lemma \ref{herm} because it suffices for the arising ray in the torus to be close to the origin, in any admissible direction. This allows us to show the result for any finite subset of the whole real line.

\begin{proof}[Proof of Theorem \ref{Hermite}] The proof is similar to that of Lemma \ref{herm}. We are only interested in finitely many points, and may thus use \eqref{simplifiedHermiteAsympt}.
Restricting attention to those $n$ which are divisible by 4 simplifies the cosine term and yields

\begin{equation}\label{HermiApprox}
\frac{\Gamma(2n+1)}{\Gamma(4n+1)}e^{-\frac{x^2}2} H_{4n}(x) 
=\cos(\sqrt{8n+1}x)+\mathcal{O}\left(\frac{1}{\sqrt{n}}\right).
\end{equation}
As before, the  statement reduces to showing that the linear flow
$$t\mapsto (a_1, a_2, \dots, a_k)t \qquad \mbox{intersects} \quad \left[-\frac{\pi}{2} + \delta, \frac{\pi}{2} - \delta \right]^k \subset \mathbb{T}^k$$ 
for an unbounded sequence of times $t_1 < t_2 < \dots$  and $\delta > 0$ which may depend on the set $\left\{a_1, a_2, \dots, a_k\right\}$. In turn, this is an immediate consequence of Lemma \ref{closeto0}, which in particular implies that any linear flow will return to, say, a
$1/10$-neighborhood of the origin infinitely often.  The cosine is positive in an entire $\pi/2$-neighborhood of the origin and the first statement follows.
By instead considering polynomials $H_n$ with $n \equiv 2~ (\mbox{mod } 4)$, we observe a phase shift in the cosine that changes the sign. The same argument applies and produces an infinite family of Hermite polynomials assuming negative values at $a_j$ for every $1\leq j\leq k$.
\end{proof}

{\it Remark.}
In the statement of Theorem \ref{Hermite}, the restriction to indices divisible by 4  is sufficient for our applications and allows to bypass a number of case distinctions. However, the argument works for every integer $n \in \mathbb{N}$, and
for linearly independent $a_1, a_2, \dots, a_k$ it implies that every possible sign pattern
appears asymptotically with density $2^{-k}$. 
Therefore, Theorem \ref{Hermite} merits further investigation only  when the points $a_1, a_2, \dots, a_k$ exhibit some form of linear dependence. 
The following example highlights  the distinguished role played by the sign configuration $(+, +, \dots, +)$.
\begin{example} \label{obstruction}
The sequence
$$ \left( H_{4n}(1),  H_{4n}(2),  H_{4n}(3),  H_{4n}(4) \right)_{n=1}^{\infty}$$
assumes the sign configuration $(+, +, -, +)$ at most finitely many times.
\end{example}
\begin{proof}[Sketch of proof]  Using \eqref{HermiApprox} and a simple expansion,
\begin{align*}
\frac{\Gamma(2n+1)}{\Gamma(4n+1)}e^{-\frac{x^2}2} H_{4n}(x) 
=\cos(\sqrt{8n+1}x)+\mathcal{O}\left(\frac{1}{\sqrt{n}}\right) 
= \cos(\sqrt{8n}x)+\mathcal{O}\left(\frac{1}{\sqrt{n}}\right).
\end{align*}
As before, this reduces the problem to studying the flow $t\mapsto(t,2t,3t,4t)$ on the torus $\mathbb{T}^4$. We would like to know that this flow intersects the subset 
$$ \left( \mathbb{T} \setminus \left[ \frac{\pi}{2}, \frac{3 \pi}{2} \right]  \right) \times
\left( \mathbb{T} \setminus \left[ \frac{\pi}{2}, \frac{3 \pi}{2} \right]  \right) \times
\left[ \frac{\pi}{2}, \frac{3 \pi}{2} \right]
\times \left( \mathbb{T} \setminus \left[ \frac{\pi}{2}, \frac{3 \pi}{2} \right]  \right)  \subset \mathbb{T}^4$$
at most finitely many times.
Introducing the fractional part $\left\{y \right\} = y - \lfloor y \rfloor$ and performing an appropriate rescaling, we  analyze the case when the first, second and fourth coordinate behave as described, i.e.
$$   \left(  \left\{y\right\} \notin [1/4, 3/4]\right) \wedge\left(  \left\{2y\right\} \notin [1/4, 3/4]\right) \wedge \left(  \left\{4y\right\} \notin [1/4, 3/4]\right).$$
This set is $1$-periodic and easily seen to be described by the condition
$$  \left\{y\right\} \in \left[ 0, \frac{1}{16} \right) \cup  \left( \frac{15}{16},1\right),$$
which in turn implies
$$  \left\{3y\right\} \in \left[ 0, \frac{3}{16} \right) \cup  \left( \frac{13}{16},1\right).$$
This set is at positive distance $1/16$ from the interval $ [1/4, 3/4]$, and so the sign configuration $(+,+,-,+)$ is never attained. The argument up to now ignored the error term of order $n^{-1/2}$. Taking it into account, one sees that the sign configuration of $\left( H_{4n}(1),  H_{4n}(2),  H_{4n}(3),  H_{4n}(4) \right)$ will be distinct from $(+,+,-,+)$ for  every sufficiently large $n$, as desired.
\end{proof}

\subsection{Laguerre polynomials} As mentioned before, results for Hermite polynomials like Theorem \ref{Hermite} and Lemma \ref{herm} hold in greater generality. We briefly discuss the case of Laguerre polynomials (see \S \ref{ssec:Laguerre}).

\begin{proposition}\label{Laguerre}
 Let $\nu>-1$ be such that
$\nu+1/2$ is not an odd integer, and let $\left\{ a_1, a_2, \dots, a_k \right\} \subset \mathbb{R}_{+}$ be a finite set of  positive reals. Then there are infinitely many $n \in \mathbb{N}$ such
that 
$$\forall~1 \leq j \leq k: \qquad \sgn(L_{n}^{\nu}(a_j)) = \sgn \left( \cos{\left(\frac{\pi}{2}\left(\nu + \frac{1}{2} \right) \right)} \right).$$
\end{proposition}

\begin{proof}[Sketch of proof] 
Using Fej\'er's formula \eqref{laguerre-asymp}, we can repeat the same reasoning as before, and reduce matters to analyzing the flow

$$t\mapsto 2(\sqrt{a_1},\sqrt{a_2},\ldots,\sqrt{a_k})t -   \frac{\pi}{2}\left(\nu + \frac{1}{2} \right) \left( 1,1,\ldots,1 \right)$$
on $\T^k$. 
As before, the first term will pass arbitrarily close to the origin infinitely many times. 
The cosine of each of the entries of the second term is nonzero precisely when $\nu+1/2$ is not an odd integer, and the result follows.
\end{proof}

\section{Extremizers}\label{sec:Structure}

\subsection{Existence of extremizers} 
 
The proof  of Theorem \ref{existence-thm} requires two results from the literature. The following lemma can be found in most functional analysis books, see \mbox{e.g. \cite{ET}}.

\begin{lemma}[Mazur's Lemma]
Let $E$ be a Banach space and let $\{x_n\}$ be a sequence in $E$ such that $x_n \rightharpoonup x$ in the weak topology. Then there exists a sequence $\{y_n\}$ in $E$, such that each $y_n$ is a convex combination of $\{f_n,f_{n+1},...,f_{N_n}\}$, for some $N_n\geq n$, and such that
$$
y_n \to x \ \ \ \  \text{strongly}.
$$
\end{lemma}

To show that extremizer candidates are nonzero, we will appeal to a higher dimensional version of the uncertainty principle of Nazarov \cite{Naz} due to Jaming \cite{J}. Since we will deal with balls only, we state the following result, which is sufficient for our purposes.

\begin{theorem}[Nazarov \& Jaming]\label{NazUP}
Let $B_1$ and $B_2$ be balls in $\R^d$ of radius $r_1$ and $r_2$ respectively. Then there exists a constant $C=C(d,r_1,r_2)$ such that, for every function $f\in L^2(\R^d)$,
$$
\int_{\R^d} |f(x)|^2dx \leq C \bigg(\int_{\R^d \setminus B_1} |f(x)|^2dx + \int_{\R^d \setminus B_2} |\widehat f(x)|^2dx\bigg).
$$
\end{theorem}

\begin{lemma}\label{uniform-est}
Let $f\in L^1(\R^d)$ satisfy $\widehat f=f$, $f(0)=\widehat f(0)=0$ and $\|f\|_{L^1}=1$. Let $B$ be a ball of radius $r$ centered at the origin, such that $\{x\in\R^d: f<0\}\subset B$. Then there exists a constant $K=K(d,r)>0$, such that
$$
\int_{B} f(x)dx \leq -K.
$$
\end{lemma}

\begin{proof}
Specializing Theorem \ref{NazUP} to $B_1=B_2=B$, yields
$$
\int_{\R^d} |f(x)|^2 dx \leq 2C(d,r) \int_{\R^d \setminus B} |f(x)|^2dx.
$$
Since $\widehat f=f$ and $\|f\|_{L^1}=1$, we have that $\|f\|_{L^\infty} \leq 1$. It follows that
$$
\int_{\R^d} |f(x)|^2 dx \leq 2C(d,r) \int_{\R^d \setminus B} |f(x)|dx = 2C(d,r) \int_{\R^d \setminus B} f(x)dx,
$$
where the last identity is due to $\{x\in\R^d: f(x)<0\}\subset B$. We also have that
$$
\frac 1 2 = \int_{\{f<0\}}|f(x)|dx \leq |B|^{1/2} \|f\|_{L^2},
$$
which in turn implies 
$$
2C(d,r) \int_{\R^d \setminus B} f(x)dx \geq \frac{1}{4|B|}.
$$
Since $\int_{\R^d}f=0$, we finally conclude 
$$
\int_{B}f(x)dx \leq -\frac{1}{8C(d,r)\omega_d r^d}.$$
\end{proof}

\begin{proof}[Proof of Theorem \ref{existence-thm}]
Let $\{f_n\}\subset \mathcal{A}$ be an extremizing sequence for inequality \eqref{sharp}. In particular, $A(\widehat f_n)A(f_n)\to {\bf A}^2$, as $n\to\infty$. By the reductions of \S \ref{sec:Reductions}, we may assume that each function $f_n$ is radial, and satisfies $\widehat f_n=f_n$, and $f_n(0)=\widehat f_n(0)=0$, and $\|f_n\|_{L^1}=1$. Extracting a subsequence if necessary, we may further assume that $\{A(f_n)\}$ is a strictly decreasing sequence, otherwise there is nothing to prove. It follows that $A(f_n)\searrow {\bf A}$. These reductions imply 
$$
\|f_n\|_{L^2}\leq 1, \textrm{ for every } n.
$$
By the Banach-Alaoglu Theorem, we may assume (again extracting a subsequence if necessary) that the sequence $\{f_n\}$ converges to some function $f\in L^2$ in the weak topology of $L^2(\R^d)$. In other words, for every function $\varphi \in L^2$, 
$$
\lim_{n\to\infty} \int_{\R^d} f_n(x)\varphi(x)dx = \int_{\R^d} f(x)\varphi(x)dx.
$$
Clearly, $\widehat f=f$. Applying Lemma \ref{uniform-est} together with the fact that the sequence $\{A(f_n)\}$ is decreasing, we see that 
$$
\int_{B(0,{A(f_1)})}f(x)dx = \lim_{n\to\infty} \int_{B(0,A(f_1))} f_n(x)dx \leq -K(d,A(f_1)).
$$
Hence $f$ is nonzero. Further note that, if $S$ is a compact set such that $S\subset \R^d\setminus \overline{B(0,{\bf A})}$, then $f_n(x)\geq 0$ for every $x\in S$, if $n$ is sufficiently large. It follows that  
$$
\int_{S} f(x)dx \geq 0,
$$
 and so $f(x)\geq 0$, for almost every $x\in\R^d\setminus \overline{B(0,{\bf A})}$. Consequently, $A(f)\leq {\bf A}$.

We claim that $f\in \mathcal{A}$, and that $f$ is an extremizer for inequality \eqref{sharp}.  Mazur's Lemma implies the existence of a sequence $\{g_n\}\subset L^2$, such that each function
$g_n$ belongs to the convex hull of $\{f_n,....,f_{N_n}\}$, for some $N_n\geq n$, and
$$
\lim_{n\to\infty} \|g_n -f\|_{L^2} = 0.
$$
Again extracting a subsequence of $\{g_n\}$ if necessary, we may assume that $g_n(x)\to f(x)$, for almost every $x\in\R^d$. Since the sequence $\{A(f_n)\}$ is decreasing, we have that $A(g_n) \leq A(f_n) \leq A(f_1)$. An application of Fatou's Lemma yields
$$
\int_{|x|\geq A(f_1)} f(x)dx = \int_{|x|\geq A(f_1)} \lim_{n\to\infty} g_n(x)dx \leq \liminf_{n\to\infty} \int_{|x|\geq A(f_1)} g_n(x)dx \leq \liminf_{n\to\infty} \|g_n\|_{L^1} \leq 1.
$$
This implies  $f\in L^1(\R^d)$. 
Now, for each $n$,  the inequalities $|f_n(x)| \leq \|f_n\|_{L^1} \leq 1$ hold, for almost every $x\in\R^d$. It follows that, for every $n$, $\|g_n\|_{L^\infty} \leq 1$,  which in turn implies  $\|f\|_{L^\infty} \leq 1$. Define the functions 

$$h_n:=g_n+\chi_{B(0,A(f_1))},\;\;\;(n\in\N).$$
These are nonnegative functions that converge pointwise almost everywhere to $f+\chi_{B(0,A(f_1))}$. Since $\widehat f_n(0)=0$, for every $n$, we have that $\widehat h_n(0)=\omega_d A(f_1)^d$, for every $n$. An application of Fatou's Lemma to the sequence $\{h_n\}$ implies $\widehat f(0) \leq 0$. We conclude that $f\in\mathcal{A}$, which in particular implies  $A(f)\geq {\bf A}$, hence $A(f)={\bf A}$, and $f$ is an extremizer. Finally, if $f(0)=\widehat f(0)<0$, then the function $g=f-f(0)e^{-\pi|\cdot|^2}$ would contradict the fact that $f$ is an extremizer. We deduce that $f(0)=0$. The proof of the theorem is now complete.
\end{proof}

\subsection{Infinitely many double roots}
As discussed in the Introduction, we split the proof of Theorem \ref{ext} in two parts. The first one works in the case $d=1$ only, and involves the sequence of functions $\{\varphi_n\}$ which was defined in \eqref{def_varphi} and studied in \S \ref{sec:flows}.
The principle at work is easy to describe: If $f$ has a finite number of double roots, then we identify an explicit function $h$ such that the function $f_\varepsilon:=f + \varepsilon h$ satisfies all the desired properties if $\varepsilon>0$
is sufficiently small, and $A(f_\varepsilon) < A(f)$ for {some small but positive} $\varepsilon $. This is illustrated in Figure \ref{fig:perturbation} below. \\

\begin{proof}[Proof of Theorem \ref{ext} for $d=1$]
Start by noting that any function $f \in \mathcal{A}$ is uniformly continuous because $\widehat{f}$ is integrable. By the same token, $\widehat{f}$ is also uniformly continuous.
Aiming at a contradiction, let $f\in\mathcal{A}$ be an extremizer of inequality \eqref{sharp} with only a finite number of double roots. Applying the dilation symmetry allows us to assume that $A(f) = A(\widehat{f})$ without changing the number of double roots. The new function $f$ has now only finitely many double roots on $({\bf A}, \infty)$. Since $A(f) = A(\widehat{f}) = {\bf A}$, we see that the continuous function $g:=f + \widehat{f}\in\mathcal{A}$ has only finitely many double roots in the interval $({\bf A}, \infty)$ (and at most as many as $f$). Moreover, it satisfies $A(g)={\bf A}$, i.e., the function $g$ is itself an extremizer.
Using Lemma \ref{herm} with $a_1 = {\bf A}$ and $a_2, a_3, \dots, a_{k}$ equal to the positive double roots of 
$g$ in $({\bf A},\infty)$, we can ensure the existence of (infinitely many, and therefore one) $n\in\N$ such that the function $\varphi_n$ satisfies
$$\varphi_n({\bf A})>0\textrm{ and }\varphi_n(a_j)>0\textrm{ for every } 2\leq j \leq k.$$
 By continuity, the function $\varphi_n$ is  positive in an open neighborhood of ${\bf A}$ and of all the double roots of $g$. Since it tends to $0$ 
as $|x| \rightarrow \infty$, it is  bounded
from below by some constant (depending on $n$), and by construction it is also positive outside a compact interval.
Therefore, if $\varepsilon>0$ is chosen sufficiently small,  then the function $g_\varepsilon:=g + \varepsilon \varphi_n$ equals its Fourier transform, belongs to the set $\mathcal{A}$, and is strictly positive on $[{\bf A},\infty)$. By continuity of $g_\varepsilon$, there exists $\delta>0$ such that the function $g_\varepsilon$ has no roots on the half-line $[{\bf A}-\delta, \infty)$, and in particular $A(g_\varepsilon)<{\bf A}$. This is the desired contradiction which completes the proof.
\end{proof}

\begin{center}
\begin{figure}[h!]
\centering
\includegraphics[width=0.5\textwidth]{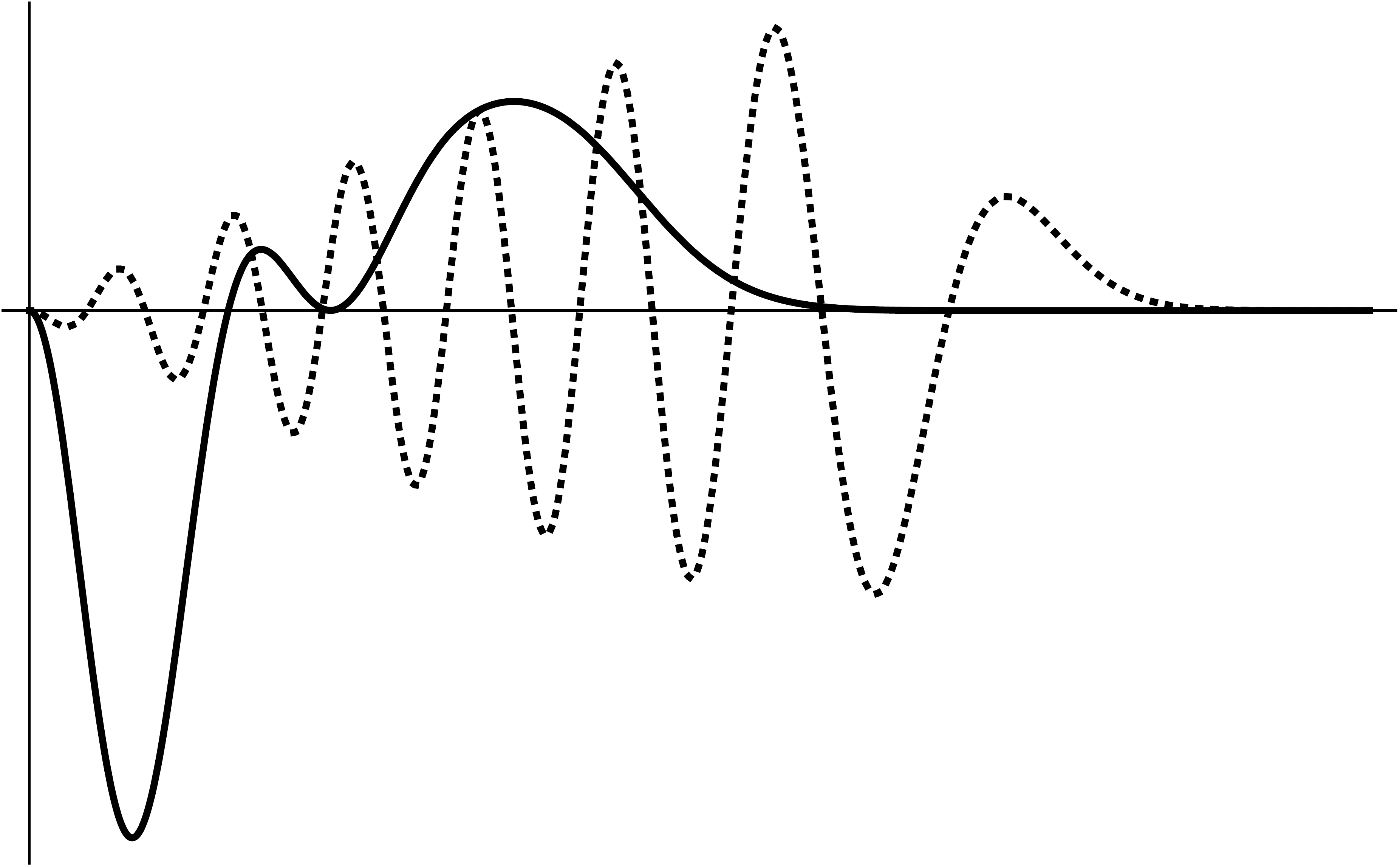}
\captionsetup{width=0.95\textwidth}
\caption{An example close to the extremizer candidate from Figure \ref{fig:goodcandidate}, having a root at $\sim 0.6$ and a unique double root at $\sim 0.9$. Adding a tiny multiple of $\varphi_6$ (dashed) moves the
root closer to the origin and resolves the double root without introducing additional roots. }
\label{fig:perturbation}
\end{figure}
\end{center}

The previous argument can be partially adapted to the higher dimensional setting, at the expense of making the construction less explicit.

\begin{proof}[Proof of Theorem \ref{ext} for $d\geq 2$]
For any $x\in\R^d$, let $r=|x|$ denote its Euclidean norm. Aiming at a contradiction, assume that $f(x)=f(r)$ is a radial extremizer of inequality \eqref{sharp} with only a finite number of double roots in the interval $(A(f),\infty)$. Applying the dilation symmetry, we can assume that $A(f) = A(\widehat{f})$ without changing the number of double roots. As a consequence, the function $f$ has only finitely many double roots on $({\bf A}, \infty)$. Similarly, we see that the continuous function $g:=f + \widehat{f}\in\mathcal{A}$ has only finitely many double roots in the interval $({\bf A}, \infty)$ (and at most as many as $f$). Moreover, it satisfies $A(g)={\bf A}$, i.e., the function $g$ is itself an extremizer. 
\smallskip

 Given any $T>\alpha>0$, we claim the existence of an integrable, radial function $\varphi:\R^d\to\R$ satisfying the following properties:
\begin{itemize}
\item[(a)] $\widehat{\varphi}=\varphi$,
\item[(b)] $\varphi(0)=0$,
\item[(c)] $\varphi(x)> 0$, for every $x$ such that $\alpha\leq |x|\leq T$,
\item[(d)] $\varphi(x)> 0$, if $|x|$ is sufficiently large.
\end{itemize}
The claim implies the existence of an admissible radial function $\varphi$, such that $\varphi=\widehat{\varphi}$ and $\varphi(x)>0$ for $|x|\in[{\bf A},T]$, where $T<\infty$ is such that  $g(x)>0$ for every $|x|>T$. The fact that $\varphi(x)> 0$ for sufficiently large values of $|x|$ implies that, for sufficiently small $\varepsilon>0$, the function $g+\varepsilon\varphi$ belongs to the class $\mathcal{A}$, and satisfies $A(g+\varepsilon\varphi)<A(g)$. This is the desired contradiction. 
\smallskip

In order to establish the claim, define the following auxiliary function:
$$
\psi_N(x;t) = \sum_{n=0}^{2N} t^nL_n^\nu(2\pi |x|^2)e^{-\pi|x|^2},\;\;\; (-1<t<1)
$$
where $\nu=d/2-1$. Note that $\widehat \psi_N(x;t) = \psi_N(x;-t)$, by property \eqref{Laguerre-multiplier}. Identity \eqref{laguerre-gen-func} also implies that

\begin{equation}\label{limpsi}
\lim_{N\to\infty} \{\psi_N(x;t) + \psi_N(x;-t)\} = \bigg(\frac{e^{-t2\pi|x|^2/(1-t)}}{(1-t)^{\nu+1}} + \frac{e^{t2\pi|x|^2/(1+t)}}{(1+t)^{\nu+1}}\bigg)e^{-\pi|x|^2},
\end{equation}
uniformly for $|x|$ in any compact subset of $(0,\infty)$. Note that \eqref{laguerre-at-zero} implies

$$\psi_N(0;t) + \psi_N(0;-t)>2, \text{ for every } t\in(-1,1) \text{ and } N>0.$$ 
Since the right-hand side of identity \eqref{limpsi} is a positive function, we conclude the existence of $A>0$ such that
$$
\eta(x):=\frac{\psi_A(x;1/2) + \psi_A(x;-1/2)}{\psi_A(0;1/2) + \psi_A(0;-1/2)} \geq \delta >0,
$$
for some $\delta>0$ and for every $|x|\in[\alpha/2,T+1]$.  
On the other hand, since $d\geq 2$ and therefore $\nu\geq 0$, identities \eqref{laguerre-asymp} and \eqref{laguerre-at-zero} imply that the sequence
\begin{equation}\label{eq-100}
\ell^\nu_{n}(x) = \binom {n+\nu} n^{-1} L_n^\nu(2\pi |x|^2)e^{-\pi|x|^2}
\end{equation}
converges to $0$, as $n\to\infty$, uniformly for $x$ in any compact subset of $(0,\infty)$. Finally, define
$$
\varphi = \eta - \ell^\nu_{2B} + \ell^\nu_{2C+2} - \ell^\nu_{2C},
$$
where $C>B$ are positive integers larger than $A$. Invoking \eqref{Laguerre-multiplier}, \eqref{laguerre-at-zero} and \eqref{high-power-laguerre}, respectively, one checks that the function $\varphi$  satisfies conditions (a), (b) and (d), for any choice of  $C>B>A$. However, since $\eta(x)\geq \delta$ for $|x|\in [\alpha/2,T+1]$, we can invoke \eqref{eq-100} in order to choose $C, B$ large enough that $\varphi(x)\geq \delta/2$, for every $|x|\in [\alpha,T]$. This shows that condition (c) is fulfilled as well, and finishes the verification of the claim. The theorem is now proved.
\end{proof}

\section{Proof of Theorem \ref{Dthm}}\label{sec:HigherDim}
This chapter improves the lower bound in all dimensions $d \geq 2$. The underlying insight is that the 
argument given in \cite{BCK} to prove Theorem \ref{BCK1} can be generalized to higher dimensions if one invokes classical properties of Bessel functions.

\subsection{Proof of the lower bound}
Let $f:\R^d\rightarrow\R$ be a  function satisfying assumptions  \eqref{nonzero_integrable}$-$\eqref{Radial}. Since $f=\widehat{f}$ and $f$ is radial, we have that, for any $x\in\R^d$,
$$f(x)=\int_{\R^d} f(y)\cos(2\pi x\cdot y)dy=\int_{\R^d} f(y)(\cos(2\pi x\cdot y)-1)dy,$$
where the last identity follows from the fact that $f$ has zero average. Writing $f=f^+-f^-$ as before, 
one has that
$$f^+(x)-f^-(x)=\int_{\R^d} (f^+(y)-f^-(y))(\cos(2\pi x\cdot y)-1)dy.$$
Equivalently,
$$f^-(x)-f^+(x)=\int_{\R^d} f^+(y)(1-\cos(2\pi x\cdot y))dx-\int_{\R^d} f^-(y)(1-\cos(2\pi x\cdot y))dy.$$
Notice that both of these integrals are positive, as are both of the summands in the left-hand side of this identity. By considering the cases  $f(x)\leq 0$ and $f(x)>0$ separately, it follows that
\begin{equation}\label{first_step}
f^-(x)\leq \int_{\R^d} f^+(y)(1-\cos(2\pi x\cdot y))dy.
\end{equation}
Now, if $f$ is radial, then so are $f^-, f^+$. In this case, one can express the right-hand side of inequality \eqref{first_step} in terms of Bessel functions. Switching to polar coordinates,
\begin{align*}
\int_{\R^d} f^+(y)(1-\cos(2\pi x\cdot y))dy
&=\int_0^\infty f^+(r) \Big( \int_{\mathbb{S}^{d-1}}(1-\cos(2\pi r x\cdot y))d\sigma(y)  \Big) r^{d-1}dr.
\end{align*}
Appealing to formula \eqref{innerproduct_int}, we see that the inner integral satisfies
$$\int_{\mathbb{S}^{d-1}}(1-\cos(2\pi r x\cdot y))d\sigma(y)
=\omega_{d-2} 
\int_{-1}^1 (1-\cos(2\pi r |x| t))(1-t^2)^{\frac{d-3}{2}}dt.$$ 
To compute the integral on the right-hand side of this expression, start by noting that
 $$\int_{-1}^1 (1-t^2)^{\frac{d-3}{2}}dt=\int_{-\frac\pi 2}^{\frac\pi 2} \cos^{d-2}(\theta)d\theta
 =\frac{\sqrt{\pi}\Gamma(\frac{d-1}2)}{\Gamma(\frac d 2)},$$
 as can be seen via repeated integration by parts. On the other hand, formula \eqref{alt_Bessel} implies that
 $$\int_{-1}^1 \cos(2\pi r |x| t)(1-t^2)^{\frac{d-3}{2}}dt
 =\frac{\sqrt{\pi}\Gamma(\frac {d-1} 2)}{(\pi r |x|)^{d/2-1}}J_{d/2-1}(2\pi r |x|).$$
 It follows that
 \begin{equation}\label{f-pointwise}
 f^-(x)
 \leq
 \omega_{d-2} \sqrt{\pi}\frac{\Gamma\Big(\frac{d-1}2\Big)}{\Gamma(d/2)} \int_0^\infty f^+(r) 
 \Big( 1 - \frac{\Gamma(\frac d 2)J_{d/2-1}(2\pi r |x|)}{(\pi r |x|)^{d/2-1}}  \Big) r^{d-1}dr.
\end{equation}
The dimensional constant appearing on the right-hand side of this inequality can be written as 
$$\omega_{d-1}=\omega_{d-2} \sqrt{\pi}\frac{\Gamma\Big(\frac{d-1}2\Big)}{\Gamma(\frac d 2)}.$$
Integrating inequality \eqref{f-pointwise} over the ball ${B_A}\subset\R^d$ centered at the origin of radius $A:=A(f)$, 
 \begin{equation}\label{second_step}
 \int_{{B_A}}f^-(x) dx
 \leq
\omega_{d-1} \int_0^\infty f^+(r) 
\Big[ \int_{{B_A}}\Big( 1 - \frac{\Gamma(\frac d 2)J_{d/2-1}(2\pi r |x|)}{(\pi r |x|)^{d/2-1}}  \Big)dx\Big] r^{d-1}dr.
\end{equation}
Since $f$ has zero average and $\|f\|_{L^1}=1$,

$$0=\int_{\R^d} f=\int_{\R^d} f^+-\int_{\R^d} f^-
\;\textrm{ and }\;
1=\int_{\R^d} |f|=\int_{\R^d} f^++\int_{\R^d} f^-.$$
It follows that 
\begin{equation*}
\int_{\R^d} f^+=\int_{\R^d} f^-=\frac1 2.
\end{equation*}
By definition of $A$, the support of the function $f^-$ is contained in the ball ${B_A}$. As a consequence, the left-hand side of inequality \eqref{second_step} equals
$$\int_{{B_A}}f^-=\int_{\R^d} f^-=\frac 1 2.$$
To handle the right-hand side, we use polar coordinates and change variables to compute
\begin{align*}
\int_{{B_A}}\Big( 1 - &\frac{\Gamma(\frac d 2)J_{d/2-1}(2\pi r |x|)}{(\pi r |x|)^{d/2-1}}  \Big)dx
=
\omega_{d-1} \int_0^A \Big( 1 - \frac{\Gamma(\frac d 2)J_{d/2-1}(2\pi r \rho)}{(\pi r \rho)^{d/2-1}}  \Big)  \rho^{d-1}d\rho\\
&=\omega_{d-1} \Big(\frac{A^d}{d}-\Gamma\Big(\frac d 2\Big)\frac{2^{d/2-1}}{(2\pi r)^{d}}\int_0^{2\pi r A}{J_{d/2-1}(s) s^{d/2}}  ds\Big)\\
&=\omega_{d-1} \frac{A^d}{d} \Big({1}- \frac{\Gamma\Big(\frac d 2+1\Big) J_{ d/2}(2\pi r A)}{(\pi r A)^{d/2}}\Big).
\end{align*}
The last identity is a consequence of Lemma \ref{BesselComp} with $\nu=d/2$ and $\rho=2\pi r A$.
Going back to \eqref{second_step}, we now have that
 \begin{equation*}
\frac1 2
 \leq
\omega_{d-1}^2\frac{A^d}{d}  \int_0^\infty f^+(r) 
 \Big({1}- \frac{\Gamma(\frac d 2+1) J_{ d/2}(2\pi r A)}{(\pi r A)^{d/2}}\Big) r^{d-1}dr.
\end{equation*}
Using H\"older's inequality and recalling that
$$\frac1 2=\int_{\R^d} f^+=\omega_{d-1}\int_0^\infty f^+(r) r^{d-1}dr$$
since $f^+$ is radial,
we have that 
$$1
 \leq
\omega_{d-1}\frac{A^d}{d} \sup_{t\in\R_+}\left|{1}- \frac{\Gamma(\frac d 2+1) J_{ d/2}(t)}{(t/2)^{d/2}}\right|.$$
This translates into
$$A^d\geq \frac d {\omega_{d-1}}\frac{1}{1+\lambda_d},$$
where
\begin{equation}\label{def_lambda}
\lambda_d:=-\inf_{t\in\R_+}{ \frac{\Gamma\left(\frac d 2+1\right) J_{ d/2}(t)}{(t/2)^{d/2}} }.
\end{equation}
Equivalently,
$$A\geq \frac{1}{\sqrt{\pi}}\left(\frac{1}{1+\lambda_d}\Gamma(\frac{d}{2}+1)\right)^{\frac{1}{d}},$$
which is clearly an improvement over the lower bound given in Theorem \ref{BCKhigherD} as long as $\lambda_d<1$. In the next section, we show that the sequence $\lambda=\{\lambda_d\}$ satisfies $\lambda_d<1/2$ for every $d\geq 2$, and that $\lambda_d\to0$ as $d\to\infty$ exponentially fast.

\subsection{Studying the sequence $\lambda$}
Define the auxiliary function
$$\Lambda_d(t):=\frac{J_{d/2}(t)}{(t/2)^{d/2}}.$$
The infimum in \eqref{def_lambda} is actually a minimum, and is attained by the first zero $t_0$ of the function $\Lambda'_d$.
This is a consequence of Lemma \ref{BesselExt}.
To find the first zero of the function $\Lambda'_d$, compute
$$\Lambda'_d(t)=\Big[J_{d/2}'(t) \Big(\frac t 2\Big)^{d/2}-J_{d/2}(t)\frac1 2\frac d 2 \Big(\frac t 2\Big)^{d/2-1}\Big] \Big(\frac t 2\Big)^{-d}.$$
It follows that  $t>0$ is a zero of the function $\Lambda'_d$ if and only if
$$J_{d/2}'(t) \Big(\frac {t} 2\Big)^{d/2}=J_{d/2}(t)\frac1 2\frac d 2 \Big(\frac {t} 2\Big)^{d/2-1},$$
or equivalently
$$2J_{d/2}'(t) =\frac d{t}J_{d/2}(t).$$
Recalling recursion relations \eqref{Recurrence1} and \eqref{Recurrence2}, this can be rewritten as
$$J_{d/2-1}(t)-J_{d/2+1}(t)=J_{d/2-1}(t)+J_{d/2+1}(t).$$
It follows that
$$t_0=j_{d/2+1},$$
where $j_{d/2+1}$ denotes the smallest positive zero of the Bessel function $J_{d/2+1}$ on the  real axis. We conclude that 
\begin{equation}\label{alt_exp_lambda}
\lambda_d=-\frac{2^{d/2}\Gamma\Big(\frac d 2+1\Big) J_{ d/2}(j_{d/2+1})}{(j_{d/2+1})^{d/2}}.
\end{equation}
{\it Mathematica} computes these values to any prescribed accuracy. For instance, with precision $5\times 10^{-3}$, we have that 

\begin{center}
\begin{tabular}{ l || c | c | c | c | c | c | c | c }
  $d$ & 2 & 3 & 4 & 5 & 6 & 7 & 8 & 9\\
  \hline
  $\lambda_d$ & 0.132 & 0.086 & 0.058 & 0.041 & 0.029 & 0.021 & 0.015 & 0.011
\end{tabular}
\end{center}

\vspace{4pt}
\noindent We conclude by showing that the sequence $\lambda$ tends to zero exponentially fast. For our purposes, it will suffice to additionally show that 

\begin{equation}\label{lambdad_geq10}
\lambda_d<\frac{1}{2} \textrm{ if }d\geq 10.
\end{equation}
Recall Stirling's formula  \eqref{Stirling} for the Gamma function and apply it to $\Gamma(\frac d 2+1)$.
It is an immediate consequence of our discussion in the proof of Lemma \ref{BesselExt} that
\begin{equation}\label{Bessel_zeros}
j_{d/2+1}> d/2+1
\end{equation}
for every $d\in\N$.
Formulas \eqref{alt_exp_lambda}, \eqref{Stirling} and \eqref{Bessel_zeros}, together with the basic estimate $|J_{d/2}|\leq 1$, imply 
\begin{align*}
0<
\lambda_d
&\leq \frac{2^{d/2}\sqrt{2 \pi} (d/2+1)^{(d/2+1)-1/2} e^{-(d/2+1)}e^{\mu(d/2+1)}}{(d/2+1)^{d/2}}\\
&\leq \frac{\sqrt{2 \pi}}{e}e^{\frac 1{6(d+2)}}\Big(\frac d 2+1\Big)^{1/2} \Big(\frac 2 e\Big)^{d/2}=:U_d.
\end{align*}
One can readily check that \eqref{lambdad_geq10} follows. 
Indeed, $U_{10}\leq 0.494$ and the sequence $\{U_d\}$ is monotonically decreasing.
Moreover, $U_d\to 0$ as $d\to\infty$ exponentially fast, and so does the sequence $\lambda$. This concludes the proof of Theorem \ref{Dthm}.
\qedhere\\

\end{document}